\documentclass{article}

\usepackage[left=3cm,right=3cm, top=3cm,bottom=2cm,bindingoffset=0cm]{geometry}
    
\usepackage{graphicx}
\usepackage{amsmath}
\usepackage{amsthm}
\usepackage{amssymb}
\usepackage[latin1]{inputenc}
\usepackage{subfigure}
\usepackage[T1]{fontenc}
\usepackage{array}
\usepackage{multicol}

\newtheorem{theorem}{Theorem}

\newtheorem{defn}{Definition}
\newtheorem{lem}{Lemma}
\newtheorem{proposition}{Proposition}
\newtheorem{corollary}{Corollary}
\newtheorem{example}{Example}

\newtheorem{remark}{Remark}

\newcommand{\ga}{\gamma}
\newcommand{\al}{\alpha}
\newcommand{\be}{\beta}

\newcommand{\la}{\lambda}
\newcommand{\eps}{\varepsilon}

\usepackage{tikz, tikz-3dplot, pgfplots}
\usepackage{bm, relsize}
\usepackage{tkz-graph}
\usetikzlibrary[positioning,patterns] 

\usepackage{pgf}
\usepackage{mathrsfs}
\usetikzlibrary{arrows}

\pgfplotsset{compat=1.8}

\tikzstyle{bsq}=[rectangle, draw, thick, minimum width=1cm, minimum height=1cm] 
\tikzstyle{bver}=[rectangle, draw, thick, minimum width=1cm, minimum height=2cm]
\tikzstyle{bhor}=[rectangle, draw, thick, minimum width=2cm, minimum height=1cm]

\tikzstyle{bsqg}=[rectangle, draw=gray!30!white, thick, fill=gray!30!white, minimum width=1cm, minimum height=1cm] 
\tikzstyle{bverg}=[rectangle, draw=gray!30!white, thick, fill=gray!30!white, minimum width=1cm, minimum height=2cm]
\tikzstyle{bhorg}=[rectangle, draw=gray!30!white, thick, fill=gray!30!white, minimum width=2cm, minimum height=1cm]

\providecommand{\keywords}[1]{\textbf{\textit{Keywords:}} #1}

\def\@address{\relax}
\def \address{\@getaddress}
\def \@getaddress#1{{
  \gdef \@address{#1}
}}
\newcommand \addressmark[1]{%
    $^{#1}$%
}
\usepackage[backend=bibtex]{biblatex}
\makeatletter
\providecommand*{\shuffle}{%
  \mathbin{\mathpalette\shuffle@{}}%
}
\newcommand*{\shuffle@}[2]{%
  \sbox0{$#1\vcenter{}$}%
  \kern .15\ht0 
  \rlap{\vrule height .25\ht0 depth 0pt width 2.5\ht0}%
  \raise.1\ht0\hbox to 2.5\ht0{%
    \vrule height 1.75\ht0 depth -.1\ht0 width .17\ht0 %
    \hfill
    \vrule height 1.75\ht0 depth -.1\ht0 width .17\ht0 %
    \hfill
    \vrule height 1.75\ht0 depth -.1\ht0 width .17\ht0 %
  }%
  \kern .15\ht0 
}
\makeatother

\begin{document}

\title{\vspace{-25pt} A domino tableau-based view on type B Schur-positivity}
\author{Alina R. Mayorova\addressmark{1}\addressmark{2} \and Ekaterina A. Vassilieva\addressmark{2}}
\date{}
\address{\addressmark{1}Department of Higher Algebra, Moscow State University, Moscow, Russia\\ \addressmark{2}Laboratoire d'Informatique de l'Ecole Polytechnique, 91128 Palaiseau Cedex, France }
\maketitle
\vspace{-16pt}
{\centering \small \itshape \@address \par}

\begin{abstract}
Over the past years, major attention has been drawn to the question of identifying Schur-positive sets, i.e. sets of permutations whose associated quasisymmetric function is symmetric and can be written as a non-negative sum of Schur symmetric functions. The set of arc permutations, i.e. the set of permutations $\pi$ in $S_n$ such that for any $1\leq j \leq n$, $\{\pi(1),\pi(2),\dots,\pi(j)\}$ is an interval in $\mathbb{Z}_n$ is one of the most noticeable examples.
This paper introduces a new type B extension of Schur-positivity to signed permutations based on Chow's quasisymmetric functions and generating functions for domino tableaux. As an important characteristic, our development is compatible with the works of Solomon regarding the descent algebra of Coxeter groups. In particular, we design descent preserving bijections between signed arc permutations and sets of domino tableaux to show that they are indeed type B Schur-positive.
\end{abstract}

\section{Background}
\subsection{Young tableaux and descent sets}
For any positive integer $n$ write $[n] = \{1,\dots, n\}$ and $S_n$ the symmetric group on $[n]$. A  {\bf partition} $\la$ of an integer $n$, denoted $\la \vdash n$  is a sequence $\la=(\la_1,\la_2,\dots,\la_p)$ of $\ell(\la)=p$ parts sorted in decreasing order such that $|\la| = \sum_i{\la_i} = n$.  A partition $\la$ is usually represented as a~Young diagram of $n=|\la|$ boxes arranged in $\ell(\la)$ left justified rows so that the $i$-th row from the top contains $\la_i$ boxes.  A Young diagram whose boxes are filled with positive integers such that the entries are increasing along the rows and strictly increasing down the columns is called a {\bf semistandard Young tableau}. If the entries of a semistandard Young tableau are restricted to the elements of $[n]$ and strictly increasing along the rows, we call it a {\bf standard Young tableau}. The partition $\la$ is the {\bf shape} of the tableau, and we denote $SYT(\la)$ (resp. $SSYT(\la)$) the set of standard (resp. semistandard) Young tableaux of shape $\la$. 
\begin{example}
\label{example : YT}
The diagrams on Figure \ref{fig : YT} are standard Young tableaux of shape $\la = (6,4,2,1,1)$.
\begin{figure} [h]
\label{fig : YT}
$$
T = \begin{matrix} 
\resizebox{!}{2.5cm}{%
\begin{tikzpicture}[node distance=0 cm,outer sep = 0pt]
	      \node[bsq] (1) at (0,  0) {\bf \Large 1};
	      \node[bsq] (2) [below = of 1] {\bf \Large 2};   
	      \node[bsq] (3) [right = of 1] {\bf \Large 3};
	      \node[bsq] (4) [right = of 3] {\bf \Large 4};	      
	      \node[bsq] (5) [below = of 2] {\bf \Large 5};   
	      \node[bsq] (6) [right = of 2] {\bf \Large 6};      
	      \node[bsq] (7) [right = of 6] {\bf \Large 7};     
	      \node[bsq] (8) [right = of 4] {\bf \Large 8};     
	      \node[bsq] (9) [right = of 8] {\bf \Large 9};   
	      \node[bsq] (10) [right = of 7] {\bf \Large 10}; 	       	      
	      \node[bsq] (11) [below = of 5] {\bf \Large 11};   
	      \node[bsq] (12) [right = of 5] {\bf \Large 12};  	      
	      \node[bsq] (13) [below = of 11] {\bf \Large 13};  
	      \node[bsq] (14) [right = of 9] {\bf \Large 14};   
\end{tikzpicture}
}
\end{matrix}
~~~
U = \begin{matrix} 
\resizebox{!}{2.5cm}{%
\begin{tikzpicture}[node distance=0 cm,outer sep = 0pt]
	      \node[bsq] (1) at (0,  0) {\bf \Large 1};
	      \node[bsq] (2) [below = of 1] {\bf \Large 2};   
	      \node[bsq] (3) [right = of 1] {\bf \Large 3};
	      \node[bsq] (4) [right = of 3] {\bf \Large 4};	      
	      \node[bsq] (5) [below = of 2] {\bf \Large 5};   
	      \node[bsq] (6) [right = of 2] {\bf \Large 6};      
	      \node[bsq] (7) [right = of 6] {\bf \Large 12};     
	      \node[bsq] (8) [right = of 7] {\bf \Large 14};     
	      \node[bsq] (9) [right = of 4] {\bf \Large 7};   
	      \node[bsq] (10) [right = of 5] {\bf \Large 13}; 	       	      
	      \node[bsq] (11) [below = of 5] {\bf \Large 10};   
	      \node[bsq] (12) [right = of 9] {\bf \Large 8};  	      
	      \node[bsq] (13) [below = of 11] {\bf \Large 11};  
	      \node[bsq] (14) [right = of 12] {\bf \Large 9};  
\end{tikzpicture}
}
\end{matrix}
$$
\caption{Two standard Young tableaux of shape $(6,4,2,1,1)$ and descent set $\{1,4,9,10,12\}$.}
\end{figure}

\end{example}
Define the {\bf descent set of a standard Young tableau} $T$ as $Des(T) = \{1\leq i \leq n-1\mid i $ is in a strictly higher row than $i+1\}$. For instance the descent set of the tableaux in Example~\ref{example : YT} is $\{1,4,9,10,12\}$. Similarly, the descent set of a permutation $\pi$ in $S_n$ is the subset of $[n-1]$ defined as $Des(\pi) = \{1\leq i \leq n-1\mid \pi(i)>\pi(i+1)\}$.\\
The {\it Robinson-Schensted (RS) correspondence} (\cite{Rob38, Sch61}) is a bijection between permutations $\pi$ in $S_n$ and ordered pairs of standard Young tableaux $(P,Q)$ of the same shape $\la \vdash n$. This bijection is decent preserving in the sense that
\begin{align*}
&Des(\pi) = Des(Q),\\
&Des(\pi^{\text{-}1}) = Des(P).
\end{align*}
\begin{example}
\label{example : RS}
Figure \ref{fig : RS} shows a permutation $\pi$ in $S_5$ and its image $(P,Q)$ by the RS correspondence. The descent preserving property reads $Des(\pi) = Des(Q) = \{1, 3\}$ and $Des(\pi^{\text{-1}}) = Des(P) = \{1,3,4\}$.
\begin{figure} [h]
\label{fig : RS}
$$
\pi = 52413
~~\xrightarrow[~~~~~~]{}~~ 
\left ( 
P = \begin{matrix} 
\resizebox{!}{1.5cm}{%
\begin{tikzpicture}[node distance=0 cm,outer sep = 0pt]
	      \node[bsq] (1) at (0,  0) {\bf \Large 1};
	      \node[bsq] (2) [below = of 1] {\bf \Large 2};   
	      \node[bsq] (3) [right = of 1] {\bf \Large 3};
	      \node[bsq] (4) [below = of 3] {\bf \Large 4};	      
	      \node[bsq] (5) [below = of 2] {\bf \Large 5};    
\end{tikzpicture}
}
\end{matrix}
,~~~
Q = \begin{matrix} 
\resizebox{!}{1.5cm}{%
\begin{tikzpicture}[node distance=0 cm,outer sep = 0pt]
	      \node[bsq] (1) at (0,  0) {\bf \Large 1};
	      \node[bsq] (2) [below = of 1] {\bf \Large 2};   
	      \node[bsq] (3) [right = of 1] {\bf \Large 3};
	      \node[bsq] (4) [below = of 3] {\bf \Large 5};	      
	      \node[bsq] (5) [below = of 2] {\bf \Large 4};    
\end{tikzpicture}
}
\end{matrix}
~
\right )
$$
\caption{A permutation and its image by the RS correspondence.}
\end{figure}

\end{example}
\subsection{Schur-positivity}
 
Let $X =\{x_1,x_2,\dots\}$ be a totally ordered set of commutative indeterminates.
Given any subset $\mathcal{A}$ of permutations in $S_n$, Gessel introduces in~\cite{Ges84} the formal power series in $\mathbb{C}[X]$:
$$
\mathcal{Q}(\mathcal{A})(X) = \sum_{\pi \in \mathcal{A}} F_{Des(\pi)}(X),
$$
where for any subset $I \subseteq [n-1]$, $F_I(X)$ is the {\bf fundamental quasisymmetric function} defined by:
\begin{equation*}
F_{I}(X) = \sum\limits_{\substack{i_1 \leq \dots \leq i_n\\k\in I \Rightarrow i_k<i_{k+1}}} x_{i_1}x_{i_2} \dots x_{i_n}.
\end{equation*}
The power series $F_{I}(X)$ is not symmetric in $X$ but verifies the property that for any strictly increasing sequence of indices $i_1 < i_2 < \dots < i_p$ the coefficient of $x_1^{k_1}  x_2^{k_2}  \dots  x_p^{k_p}$ is equal to the coefficient of $x_{i_1}^{k_1}  x_{i_2}^{k_2}  \dots  x_{i_p}^{k_p}$. In~\cite{GesReu93} Gessel and Reutenauer looked at the problem of characterising the sets $\mathcal{A}$ for which $\mathcal{Q}(\mathcal{A})$ is symmetric. Further the question of determining {\bf Schur-positive} sets, i.e. the sets $\mathcal{A}$  for which $\mathcal{Q}(\mathcal{A})$ can be expanded with non-negative coefficients in the Schur basis received significant attention.\\ Classical examples of Schur-positive sets include inverse descent classes, Knuth classes \cite{Ges84} and conjugacy classes \cite{GesReu93}. As a more sophisticated example, Elizalde and Roichman proved \cite{EliRoi14} the Schur-positivity of {\bf arc permutations}, i.e. the set $\mathcal{A}_n$ of permutations $\pi$ in $S_n$ such that for any $1\leq j \leq n$, $\{\pi(1),\pi(2),\dots,\pi(j)\}$ is an interval in $\mathbb{Z}_n$. Arc permutations are alternatively defined as the set of permutations in $S_n$ avoiding the patterns $\sigma$ in $S_4$ such that $|\sigma(1) -\sigma(2)| = 2$. Other advanced examples of Schur-positive sets can be found in \cite{EliRoi17}. Many of these results are the consequence of two main facts.
\begin{enumerate}
\item Denote $s_\lambda$ the Schur symmetric functions indexed by $\la \vdash n$. $s_\lambda$ is the generating function for semistandard Young tableaux of shape $\la$. It follows (see e.g. \cite[7.19.7]{Sta01}) that 
\begin{align}
\label{eq : Sdecomp} s_\la(X) &=\sum_{T \in SSYT(\la)}X^T = \sum_{T \in SYT(\la)}F_{Des(T)}(X).
\end{align}
\item There are various descent preserving bijections relating sets of permutations and standard Young tableaux, e.g. the RS correspondence.
\end{enumerate}
The proof of Elizalde and Roichman in \cite{EliRoi14} also uses Equation (\ref{eq : Sdecomp}) and relies on a custom bijection between arc permutations and standard Young tableaux of shapes $(n-k,1^k)$ and $(n-k,2,1^{k-2})$. As a result they get the following equation for the quasisymmetric function of arc permutations. 
\begin{equation*}
\sum_{\pi \in \mathcal{A}_n} F_{Des(\pi)} = s_{(n)}+s_{(1^n)}+2\sum_{1\leq k \leq n-2}s_{(n-k,1^k)}+\sum_{2\leq k \leq n-2}s_{(n-k,2,1^{k-2})}.
\end{equation*}

A {\bf type B extension} of Schur-positivity deals with $B_n$, the {\bf hyperoctahedral group} of order $n$ instead of $S_n$. $B_n$ is composed of all permutations $\pi$ on $\{\text{-}n, \dots,\text{-}2, \text{-}1, 0, 1, 2, \dots, n\}$ such that $\pi(-i) = -\pi(i)$ for all $i$. Such permutations usually referred to as {\bf signed permutations} are fully described by their restriction to $[n]$. To extend items 1 and 2 above, two options are available and depend on the definition for the descent of signed permutations.\\
As a first approach, Adin et al. in \cite{AdiAthEliRoi15} use the notion of {\bf signed descent set}, i.e. the ordered pair $(S,\eps)$ defined for $\pi \in B_n$ as 
$$
S = \{n\}\cup \{1\leq i \leq n-1\mid \begin{cases}  \pi(i) > \pi(i+1), & \mbox{ if } \pi(i) >0 \\   \mbox{either } \pi(i+1) >0 \mbox{ or }  |\pi(i)| > |\pi(i+1)|, & \mbox{ if } \pi(i)<0 \end{cases} \}
$$
  and $\eps$ is the mapping from $S$ to $\{-,+\}$ defined as $\eps(s) = +$ if $\pi(s)>0$ and $\eps(s) = -$, otherwise. There is a signed descent preserving analogue of the RS correspondence relating signed permutations and {\it bi-tableaux}, i.e. ordered pairs of Young tableaux with specific constraints and \cite{AdiAthEliRoi15} proves an analogue of Equality (\ref{eq : Sdecomp}) between their generating function and {\it Poirier's signed quasisymmetric functions}. \\
While the authors succeed in extending most of the results known in type A, this definition of the descent set of a permutation is not directly conform to the work of Solomon in \cite{Sol76} on the descent algebra of Coxeter groups. As a result providing another framework relying on a more intuitive definition of descent conform to the one of Solomon in the case of signed permutations appears as a natural question. We use the following definition of the {\bf descent set} of $\pi \in B_n$ as the subset of $\{0\}\cup[n-1]$ equal to 
$$
Des(\pi) = \{0 \leq i \leq n-1 \mid \pi(i) > \pi(i+1)\}.
$$ 
Note that the main difference with respect to the case of the symmetric group is the possible descent in position $0$ when $\pi(1)$ is a negative integer. A bijection by Barbash and Vogan \cite{BarVog82} provides a descent preserving analogue of the RS correspondence that relates signed permutations and {\it domino tableaux} (see next section).\\
In Section \ref{sec : typeB}, we use our {\it modified generating function for domino tableaux} \cite{MayVas19_2} and {\it Chow's type B quasisymmetric functions} \cite{Cho01} to develop this alternative type B extension of Schur-positivity. In Section \ref{sec : arcB}, we introduce a new descent preserving bijection between {\it signed arc permutations} and domino tableaux to show that the former set is type B Schur-positive according to the definition of descent stated above. Finally, Section \ref{sec : AltProof} aims at showing the connections between our approach and the one of \cite{AdiAthEliRoi15}. 
\section{A new definition of type B Schur-positivity based on Chow's quasisymmetric functions and domino functions}
 \label{sec : typeB}
\subsection{Domino tableaux} 
For $\la \vdash 2n$, a {\bf standard domino tableau} $T$ of shape $\la$ is a Young diagram of shape $shape(T)=\la$ tiled by {\bf dominoes}, i.e. $2\times1$ or $1\times2$ rectangles filled with the elements of $[n]$ such that the entries are strictly increasing along the rows and down the columns. In the sequel we consider only the set $\mathcal{P}^0(n)$ of {\bf empty $\bf 2$-core partitions} $\la \vdash 2n$ that fit such a tiling. A standard domino tableau $T$ has a descent in position $i>0$ if $i+1$ lies strictly below $i$ in $T$ and has a descent in position $0$ if the domino filled with $1$ is vertical. We denote $Des(T)$ the set of all its descents.\\ 
We call a {\bf semistandard domino tableau} T of shape $\la \in \mathcal{P}^0(n)$ and weight $w(T) = \mu =(\mu_0, \mu_1, \mu_2,\dots)$ with $\mu_i \geq 0$ and $\sum_i \mu_i = n$ a tiling of the Young diagram of shape $\la$ with horizontal and vertical dominoes labelled with integers in $\{0,1,2,\dots\}$ such that labels are non-decreasing along the rows, strictly increasing down the columns and exactly $\mu_i$ dominoes are labelled with $i$. 
If the top leftmost domino is vertical, it cannot be labelled $0$. Denote $SDT(\la)$ (resp. $SSDT(\la)$) the set of standard (resp. semistandard) domino tableaux of shape $\la$.\\ 

\begin{remark}
Our definition of semistandard domino tableaux differs from the classical one  by the addition of '$\it 0$' entries to the domino tableaux in some cases. We need this modification to connect their generating functions to Chow's type B quasisymmetric functions and get our type B analogue of Equation (\ref{eq : Sdecomp}). 
\end{remark}
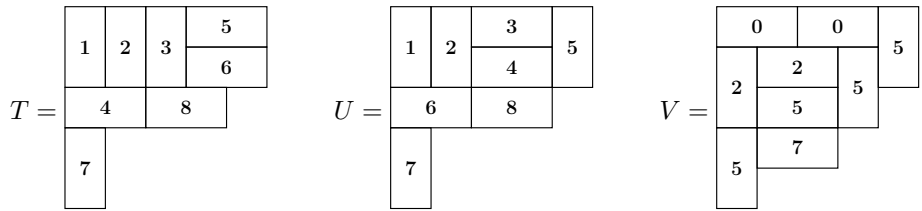
\begin{figure} [h]
\begin{center}
$$
T =
\begin{matrix}
\resizebox{!}{2.7cm}{%
\begin{tikzpicture}[node distance=0 cm,outer sep = 0pt]
        \node[bver] (1) at ( 0,   0) {\bf \Large 1};
        \node[bver] (2) [right = of 1] {\bf \Large 2};
        \node[bver] (3) [right = of 2] {\bf \Large 3};         
        \node[bhor] (4) at ( 0.5,   -1.5) {\bf \Large 4}; 
        \node[bhor] (5) at ( 3.5,   0.5) {\bf \Large 5};        
        \node[bhor] (6) [below = of 5] {\bf \Large 6};      
        \node[bver] (7) at ( 0,   -3) {\bf \Large 7};       
        \node[bhor] (8) [right = of 4] {\bf \Large 8};
\end{tikzpicture}  
}
\end{matrix}
~~~~~~~
U = 
\begin{matrix}
\resizebox{!}{2.7cm}{%
\begin{tikzpicture}[node distance=0 cm,outer sep = 0pt]
        \node[bver] (1) at ( 0,   0) {\bf \Large 1};
        \node[bver] (2) [right = of 1] {\bf \Large 2};
        \node[bhor] (3) at ( 2.5,   0.5) {\bf \Large 3};           
        \node[bhor] (4) [below = of 3] {\bf \Large 4}; 
        \node[bver] (5) at ( 4,   0) {\bf \Large 5};        
        \node[bhor] (6) at ( 0.5,   -1.5) {\bf \Large 6};  
        \node[bver] (7) at ( 0,   -3) {\bf \Large 7};      
        \node[bhor] (8) [right = of 6] {\bf \Large 8};
\end{tikzpicture}  
}
\end{matrix}
~~~~~~~
V = 
\begin{matrix}
\resizebox{!}{2.7cm}{%
\begin{tikzpicture}[node distance=0 cm,outer sep = 0pt, line width=1pt]
        \node[bhor] (1) at ( 0,   0) {\bf \Large 0};
        \node[bhor] (2) [right = of 1] {\bf \Large 0};
        \node[bver] (3) at ( -0.5,   -1.5) {\bf \Large 2};
        \node[bhor] (4) at ( 1,   -1) {\bf \Large 2};     
        \node[bver] (5) [below = of 3] {\bf \Large 5};    
        \node[bhor] (6) [below = of 4] {\bf \Large 5};       
        \node[bver] (7) at ( 2.5,   -1.5) {\bf \Large 5};    
        \node[bver] (8) at ( 3.5,   -0.5) {\bf \Large 5};           
        \node[bhor] (9) [below = of 6] {\bf \Large 7};
\end{tikzpicture}  
}
\end{matrix}
$$
\end{center}
\caption{
Two standard domino tableaux $T$ and $U$ of shape $(5,5,4,1,1)$ and descent set \{0,3,5,6\} and a semistandard domino tableau $V$ of shape $(5,5,4,3,1)$ and weight $\mu=(2,0,2,0,0,4,0,1)$. 
}
\label{fig : DTs}
\end{figure}

In \cite{MayVas19_2}, we introduce a variant of the generating function  for semistandard domino tableaux taking into account the zero values called {\bf domino function}. 
\begin{defn}[Domino functions]
Given an alphabet $X$ and a semistandard domino tableau $T$ of weight $\mu$, denote $X^T$ the monomial $x_0^{\mu_0}x_1^{\mu_1}x_2^{\mu_2}\dots$. For $\la \in \mathcal{P}^0(n)$ we call the {\bf domino function} indexed by $\la$ the function defined in the alphabet $X$ by
\begin{equation}
\mathcal{G}_\la(X)  = \sum\limits_{T \in SSDT(\la)}{X}^{T}.
\end{equation}
\end{defn}

Finally, there is a natural analogue of the RS-correspondence for signed permutations involving domino tableaux. Indeed, 
Barbash and Vogan (\cite{BarVog82}) built a bijection between signed permutations of $B_n$ and ordered pairs of standard domino tableaux of equal shape in $\mathcal{P}^0(n)$. An independent development on the subject is due to Garfinkle in \cite{Gar90, Gar92, Gar93}. Van Leeuwen shows in \cite{Lee96} that the two approaches are actually equivalent. See also Stanton and White in \cite{StaWhi85} for a more general treatment of {\it rim hook tableaux}.
Ta\c{s}kin (\cite[Prop. 26]{Tas12}) shows that the two standard domino tableaux associated to a signed permutation $\pi$ by the algorithm of Barbash and Vogan have respective descent sets $Des(\pi^{-1})$ and $Des(\pi)$. Finally, Shimozono and White prove in \cite{ShiWhi01} that half the total number of vertical dominoes in the ordered pair of domino tableaux is equal to the number of negative entries in the signed permutation (the {\bf color-to-spin} property).

\begin{example}
Figure \ref{fig : BV} shows the image $(P,Q)$ of the signed permutation $\pi = \text{-}3~8~5~\text{-}2~1~\text{-}9~\text{-}7~4~\text{-}6$ by the Barbash and Vogan bijection. One can check that $Des(\pi) = Des(Q) = \{0,2,3,5,8\}$ and $Des(\pi^{\text{-}1}) = Des(P) = \{1,4,5,8\}$. Note that the total number of vertical dominoes in $P$ and $Q$ is equal to $10$, i.e. twice the number of negative entries in $\pi$.  
\end{example}
\begin{figure} [h]
\begin{center}
$$
\pi = \text{-}3~8~5~\text{-}2~1~\text{-}9~\text{-}7~4~\text{-}6
~~\xrightarrow[~~~~~~]{}~~ 
\left ( 
P =
\begin{matrix}
\resizebox{!}{2.3cm}{%
\begin{tikzpicture}[node distance=0 cm,outer sep = 0pt]
        \node[bhor] (1) at ( 0,   0) {\bf \Large 1};
        \node[bhor] (2) [below = of 1] {\bf \Large 2};
        \node[bver] (3) at ( 1.5,   -0.5) {\bf \Large 3};         
        \node[bhor] (4) at ( 3,   0) {\bf \Large 4}; 
        \node[bhor] (5) [below = of 4] {\bf \Large 5};        
        \node[bver] (6) at ( -0.5,   -2.5) {\bf \Large 6};      
        \node[bver] (7)  [right = of 6] {\bf \Large 7};       
        \node[bver] (8) [right = of 7] {\bf \Large 8};
        \node[bhor] (9) at ( 0,   -4) {\bf \Large 9};
\end{tikzpicture}  
}
\end{matrix}
,~~~
Q = 
\begin{matrix}
\resizebox{!}{2.3cm}{%
\begin{tikzpicture}[node distance=0 cm,outer sep = 0pt]
         \node[bver] (1) at ( 0,   0) {\bf \Large 1};
        \node[bhor] (2) at ( 1.5,   0.5) {\bf \Large 2};
        \node[bhor] (3) [below = of 2] {\bf \Large 3};         
        \node[bhor] (4) at ( 0.5,   -1.5) {\bf \Large 4}; 
        \node[bver] (5) at ( 3,   0)  {\bf \Large 5};        
        \node[bver] (6) at ( 0,   -3) {\bf \Large 6};      
        \node[bver] (7)  [right = of 6] {\bf \Large 7};       
        \node[bver] (8) [right = of 5] {\bf \Large 8};
        \node[bver] (9) at ( 2,   -2) {\bf \Large 9};
\end{tikzpicture}  
}
\end{matrix}
~~\right )
$$
\end{center}
\caption{
A signed permutation and the associated ordered pair of domino tableaux. 
}
\label{fig : BV}
\end{figure}
\subsection{Chow's type B quasisymmetric functions}

Chow defines in \cite{Cho01} an analogue of Gessel's algebra of quasisymmetric functions that is dual to the Solomon's descent algebra of type B.
Let $X = \{x_0, x_1,\dots, x_i,\dots\}$ be a set of totally ordered commutative indeterminates  and $I$ be a subset of $\{0\} \cup [n-1]$, he defines a type B analogue of the fundamental quasisymmetric functions 
\begin{eqnarray*}
F_I^B(X) =  \sum_{\substack{0 = i_0\leq i_1\leq i_2\leq \ldots \leq i_n\\ j\in I \Rightarrow i_j < i_{j+1}}} x_{i_1}x_{i_2}\ldots x_{i_n}.
\end{eqnarray*}
Note the particular r\^ole of the variable $x_0$. 
\begin{example}
Let $n=2$ and $X= \{x_0, x_1, x_2\}$ then 
\begin{align*}
&F_{\emptyset} =x_0^2 + x_1^2 + x_2^2 + x_0x_1+x_0x_2+x_1x_2,\\ &F_{\{1\}} = x_0x_1 +x_0x_2+ x_1x_2,\\
&F_{\{0\}} = x_1^2 + x_2^2 +x_1x_2,\\
&F_{\{0,1\}} = x_1x_2.
\end{align*}
\end{example}

In \cite{MayVas19_2}, we show that Chow's type B fundamental quasisymmetric functions are related to our generating functions for domino tableaux. 
\begin{proposition}[\cite{MayVas19_2}]
Given an alphabet $X = \{x_0, x_1,\dots, x_i,\dots\}$, an integer $n$ and an empty $2$-core partition $\la \in \mathcal{P}^0(n)$, one can expand the domino function $\mathcal{G}_\la$ in terms of type B fundamental quasisymmetric functions as
\begin{equation}
\label{eq : GdX}
\mathcal{G}_\la(X)  = \sum_{T \in SDT(\la)} F^B_{Des(T)}.
\end{equation} 
\end{proposition}
\subsection{Type B Schur positivity}
Similarly to the type A case, given any subset $\mathcal{B}$ of $B_n$ we look at the Chow's type B quasisymmetric function 
$$\mathcal{Q}(\mathcal{B})(X) = \sum_{\pi \in \mathcal{B}} F^B_{Des(\pi)}(X).$$
Next we proceed with our definition of type B Schur positivity.
\begin{defn}[Type B Schur positivity]
\label{def : SBP}
A set $\mathcal{B} \subset B_n$ is said to be {\bf type B Schur positive} or for short {\bf $\mathcal{G}$-positive} if the function $\mathcal{Q}(\mathcal{B})$ can be written as a non-negative sum of domino functions. 
\end{defn}
This definition seems rather natural with regard to the case of type A as Equation (\ref{eq : GdX}) is a type B analogue of Equation (\ref{eq : Sdecomp}) and there are descent preserving bijections (e.g. Barbash and Vogan) between domino tableaux and signed permutations. Nevertheless, it raises the following remarks.
\begin{remark}
\label{rem : Bsym}
For a non-negative integer $n$, denote $\Lambda_n[X]$ the ring of symmetric functions of degree $n$. According to our definition of type B Schur positivity, the fact that for $\mathcal{B} \subset B_n$ the quasisymmetric function $\mathcal{Q}(\mathcal{B})$ is $\mathcal{G}$-positive does not imply that it belongs to $\Lambda_n[X]$. It actually belongs to the vector space $\Lambda^B_n[X]$ spanned by the functions of the form $x_0^k f(X^*)$ where $X^*=X\setminus\{x_0\}$, $k\leq n$ is a non-negative integer and $f$ is any symmetric function of $\Lambda_{n-k}[X^*]$. Namely, if $\mathcal{Q}(\mathcal{B})$ is $\mathcal{G}$-positive then $$\mathcal{Q}(\mathcal{B}) \in \Lambda^B_n[X]=\sum_{k=0}^n  x_0^{k}\Lambda_{n-k}[X^*].$$
\end{remark}
\begin{proof}
The proof of the statements in Remark \ref{rem : Bsym} relies on the expression of the domino functions in terms of Schur symmetric functions. There is a well known weight preserving (but not descent preserving) bijection between semistandard domino tableaux and bi-tableaux (often attributed to Littlewood \cite{Lit51} with a simpler description in \cite[Algorithm 6.1]{CarLec95}, see Section \ref{sec : BiT} for more details). 
The respective shapes of the two Young tableaux depend only on the shape of the initial semistandard domino tableau. 
Denote $(T^-,T^+)$ the bi-tableau associated to a semistandard domino tableau $T$ of shape $\la \in \mathcal{P}^0(n)$ and $(\la^-,\la^+)$ their respective shapes. We have $|\la^-| + |\la^+| = n$ and $(\la^-,\la^+)$ is called the {\bf $2$-quotient} of $\la$. The partitions $\la^-$ and $\la^+$ depend only on $\la$ and not on the entries of $T$. We show in \cite{MayVas19_2} that
\begin{align}
\label{eq : Gss}
\mathcal{G}_\la(X) = s_{\la^{-}}(X^*)s_{\la^{+}}(X).
\end{align}
Then one can compute
\begin{align*}
\mathcal{G}_\la(X) &= \sum_{ \nu;\;\la^+/ \nu \mbox{ is a horizontal strip}}s_{\la^{-}}(X^*)s_{\nu}(X^*)s_{\la^+/\nu}(x_0),\\
\mathcal{G}_\la(X) &= \sum_{\nu;\;\la^+/\nu \mbox{ is a horizontal strip}}s_{\la^{-}}(X^*)s_{\nu}(X^*)x_0^{|\la^+|-|\nu|},\\
\mathcal{G}_\la(X) &= \sum_{ \nu,\rho;\;\la^+/ \nu \mbox{ is a horizontal strip};\;\rho \vdash |\la^-|+|\nu|}k^\rho_{\la^- \nu}s_{\rho}(X^*)x_0^{|\la^+|-|\nu|},
\end{align*}
where $k^\al_{\be \ga}$ is the {\bf Littlewood-Richardson coefficient} indexed by the three partitions $\al, \be, \ga$,  i.e. the structure constants of the algebra of symmetric functions in the Schur basis and a {\bf horizontal strip} is a skew shape composed of boxes, none of which are in the same column.\\
Denote for a partition $\rho$ and a non-negative integer $n$ with $|\rho|\leq n$ the {\bf type B Schur symmetric function} $s_{(n-|\rho|,\rho)}^B$ defined by $$s_{(n-|\rho|,\rho)}^B(X)=x_0^{n-|\rho|}s_\rho(X^*).$$
One has 
\begin{equation}
\label{Gas}
\mathcal{G}_{\la}(X) = \sum_{|\rho| \leq n}a_{\rho}^\la s_{(n-|\rho|,\rho)}^B(X),
\end{equation} 
where $a_{\rho}^\la$ is a non-negative integer equal to
$$
a_{\rho}^\la = \sum_{\nu \vdash |\rho|-|\la^-|;\;\la^+/\nu \mbox{ is a horizontal strip}}k^\rho_{\la^- \nu}.
$$ 
Finally, note that the family $\{s_{(n-|\rho|,\rho)}^B(X)\}_{|\rho| \leq n}$ is a basis of $\Lambda^B_n[X]$. 
\end{proof}

\begin{remark}
The set of domino functions $\{\mathcal{G}_{\la}\}_{\la \in \mathcal{P}^0(n)}$ is not a basis of $\Lambda^B_n[X]$. Indeed, denote $p(n)$ the number of partitions of integer $n$. One has
$$
|\{\mathcal{G}_{\la}\}_{\la \in \mathcal{P}^0(n)}| = |\mathcal{P}^0(n)| = \sum_{i=0}^n p(n-i)p(i)> \sum_{i=0}^n p(i)= \dim(\Lambda^B_n[X]).
$$
As a result, there might be more than one way to decompose a function in $\Lambda^B_n[X]$ as a sum of domino functions. However, one can show that the subfamily $\{\mathcal{G}_{\la^-,(k)}\}_{k\leq n, \la^- \vdash n-k}$ is a basis of $\Lambda^B_n[X]$ where by abuse of notation we denoted $\la^-,(k)$ the partition in $\mathcal{P}^0(n)$ whose $2$-quotient is equal to $(\la^-,(k))$. We do not use this subfamily for our definition of type B Schur positivity as even the most basic examples (see next subsection) are no longer $\mathcal{G}$-positive. Secondly if a function $\mathcal{Q}(\mathcal{B})$ is $\mathcal{G}$-positive, then it can also be decomposed with non-negative coefficients in the basis $s_{(n-|\rho|,\rho)}^B$ using Equation (\ref{Gas}). 
\end{remark}
\subsection{Examples of type B Schur-positivity}
We proceed with some examples of sets of signed permutations that are $\mathcal{G}$-positive according to our definition \ref{def : SBP}.
\begin{proposition}[Inverse Descent Classes]
\label{prop : IDC} 
Let $J \subset \{0\} \cup [n]$. The inverse descent class $D^{B, -1}_{n, J} = \{\pi \in B_n~|~Des(\pi^{-1}) = J\}$ is $\mathcal{G}$-positive.
\end{proposition}
\begin{proof}
According to the Barbash and Vogan bijection \cite{BarVog82}, there is a descent preserving bijection between permutations $\pi$ of the hyperoctahedral group and ordered pairs of standard domino tableaux $(P, Q)$ of the same shape that verify $Des(\pi) = Des(Q)$ and $Des(\pi^{-1}) = Des(P)$. 
We use this property to compute $\mathcal{Q}(D^{B, -1}_{n, J})$.
\begin{align*}
\mathcal{Q}(D^{B, -1}_{n, J}) = 
\sum_{\substack{\pi \in B_n \\ Des(\pi^{-1}) = J}} F_{Des(\pi)} =  
\sum_{\la \vdash n} ~\sum_{\substack{P \in SDT(\la)\\ Des(P) = J}} ~\sum_{Q \in SDT(\la)} F^B_{Des(Q)}.
\end{align*}
Since $\mathcal{G}_{\la} = \sum_{Q \in SDT(\la)} F^B_{Des(Q)}$ we get that $\mathcal{Q}(D^{B, -1}_{n, J})$ expands in domino functions with non-negative coefficients.
\end{proof}

Another essential example of a $\mathcal{G}$-positive set is the type B analogue of {\it Knuth classes}. Denote $(P_{\pi}, Q_{\pi})$ the ordered pair of standard domino tableaux that is the image of the signed permutation $\pi$ by the Barbash and Vogan bijection \cite{BarVog82}.  
Given a standard domino tableau $T$ we denote $\mathcal{C}^B_T$ the {\bf type B Knuth class} defined by
$$\mathcal{C}^B_T = \{\pi \in B_n~|~P_{\pi} = T\}.$$
\begin{proposition}
Let $T$ be a standard domino tableau. The type B Knuth class $\mathcal{C}^B_T$ is $\mathcal{G}$-positive.
\end{proposition}
\begin{proof}
Compute $\mathcal{Q}(\mathcal{C}^B_T)$ in a similar fashion as in the proof of Proposition \ref{prop : IDC}.
\begin{align*}
\mathcal{Q}(\mathcal{C}_T) = 
\sum_{\substack{\pi \in B_n \\ P_{\pi} = T}} F_{Des(\pi)} =  
\sum_{Q \in SDT(shape(T))} F^B_{Des(Q)} = 
\mathcal{G}_{shape(T)}.
\end{align*}
That yields the desired result.
\end{proof}
The final example is a consequence of the previous ones.
We say that a permutation $\pi \in B_n$ is {\bf left-unimodal} if there exists an integer $i \in [n]$ such that 
$$\pi^{-1}(1) > \pi^{-1}(2) > \dots > \pi^{-1}(i) < \pi^{-1}(i+1) < \dots < \pi^{-1}(n).$$
\begin{proposition}
The set of left-unimodal signed permutations is $\mathcal{G}$-positive.  
\end{proposition}
\begin{proof}
A permutation $\pi$ is left-unimodal if and only if $Des(\pi^{-1}) = \{1,2,\dots,i\}$ or $Des(\pi^{-1}) = \{0, 1,2,\dots,i\}$ for some $i \in [n-1]$. As a result, the set of left-unimodal permutations is the union of inverse descent classes. As the definition of the type B quasisymmetric function of a set of permutations is additive, the result follows.	
\end{proof}
\section{Application to signed arc permutations}
\label{sec : arcB}
\subsection{Main theorem}
As stated in introduction, the set of arc permutations is a remarkable Schur positive set. We build a new bijection for the set of signed arc permutations and prove that it is type B Schur positive. 
\begin{defn}[Signed arc permutations]
\label{def : SAP}
A permutation $\pi \in B_n$ is called {\bf a signed arc permutation} if for $1\leq i\leq n$ the set 
\begin{enumerate}
\item $\{|\pi(1)|, |\pi(2)|, \dots |\pi(i-1)|\}$ is an interval in $\mathbb{Z}_{n}$ and
\item $\pi(i)>0$ if $|\pi(i)| - 1 \in \{|\pi(1)|, |\pi(2)|, \dots |\pi(i-1)|\}$ and \item $\pi(i)<0$ if $|\pi(i)| + 1 \in \{|\pi(1)|, |\pi(2)|, \dots |\pi(i-1)|\}$ (with the addition in $\mathbb{Z}_{n}$). 
\end{enumerate}
The set of signed arc permutations is denoted by $\mathcal{A}^s_n$. 
\end{defn}
\begin{remark}
As shown in \cite{EliRoi15}, signed arc permutations are exactly those permutations of $B_n$ that avoid the following 24 patterns:
$$[\pm 1, -2, \pm 3], [\pm 1, 3, \pm 2], [\pm 2, -3, \pm 1], [\pm 2, 1, \pm 3], [\pm 3, -1, \pm 2], [\pm 3, 2, \pm 1].$$ 
\end{remark}
\noindent The main result of this section follows.
\begin{theorem}
\label{thm:SAP}
The set of signed arc permutations $\mathcal{A}^s_n$ is $\mathcal{G}$-positive. Moreover,
\begin{align}
\label{eq:SAP}
\notag\sum_{\pi \in \mathcal{A}^s_{n}} F^B_{Des(\pi)} &= \mathcal{G}_{(2n)} + \mathcal{G}_{(2n-1, 1)} + \mathcal{G}_{(2n-2, 1, 1)} + \mathcal{G}_{(2n-3, 1, 1, 1)}+  2 \sum_{a \ge 2n-a \ge 2}  \mathcal{G}_{(a, 2n-a)}   \\
 &+ \sum_{a \ge 2n-a-2 \ge 2} \mathcal{G}_{(a, 2n-a-2, 2)} + \sum_{a \ge 2n-a-2 \ge 2}\mathcal{G}_{(a, 2n-a-2, 1, 1)}.
\end{align}
\end{theorem}
\noindent To prove Theorem~\ref{thm:SAP} we introduce a new descent preserving bijective map from $\mathcal{A}^s_{n}$ to the sets of standard domino tableaux with shapes equal to the indices of the domino functions in formula~\ref{eq:SAP}.

\subsection{A new description of signed arc permutations}
Our bijective proof of Theorem~\ref{thm:SAP} relies on a characterisation of signed arc permutations as a shuffle of a positive and a negative subsequence. More precisely, a permutation $\pi \in B_n$ may be written as a sequence $$\pi = \pi_1 \pi_2 \dots \pi_n$$ where $\pi_i = \pi(i)$. We denote
\begin{align*}
\pi^+ &= a_1a_2\dots a_k\\
\pi^- &= b_1b_2\dots b_l
\end{align*}
the two subsequences of $\pi$ ($k+l = n$) composed respectively of the positive and negative integers in $\pi$. According to Definition \ref{def : SAP}, if $\pi$ is a signed arc permutation, then one has 
\begin{equation}
\label{eq : subsequences}
\begin{cases} 
\{|a_i|\}_{1\leq i \leq k }\cup\{|b_j|\}_{1\leq j \leq l } = [n]\\
a_{i+1} = a_{i}+1\mbox{ for } i=1\ldots k-1\\
b_{j+1} = b_{j}+1 \mbox{ for } j=1\ldots l-1\\
a_1 = -b_1 +1 \mbox{ if } kl \neq 0
\end{cases}
\end{equation}
where the $+$ sign denotes the addition in $\mathbb{Z}_n$ subject to the cyclic conditions $n+1 = 1$ and $\text{-}1+1 = \text{-}n$.
\begin{example}
The signed arc permutation in $B_7$, $\pi = \text{-}3\text{\phantom{-}}4\text{\phantom{-}}5\text{-}2\text{-}1\text{-}7\text{\phantom{-}}6$  is composed of the two subsequences
$\pi^+ = \text{\phantom{-}}4\text{\phantom{-}}5\text{\phantom{-}}6$ and 
$\pi^- = \text{-}3\text{-}2\text{-}1\text{-}7$
\end{example}
Given two sequences $\alpha$ and $\beta$ let $\alpha \shuffle \beta$ be the set of new sequences obtained by shuffling the letters of $\alpha$ and $\beta$ such that the initial order of the letters of $\alpha$ (resp. of $\beta$) is preserved. 
\begin{proposition}
The set of signed arc permutations $\mathcal{A}_n^s$ is equal to the disjoint union of the sets of sequences $\alpha \shuffle \beta$ where $(\al,\beta)$ runs over all the ordered pairs of positive and negative sequences fulfilling the conditions \ref{eq : subsequences}.
\end{proposition}
\begin{proof}
This proposition is a direct consequence of Definition \ref{def : SAP} and the description of arc permutations in terms of positive and negative subsequences.
\end{proof}
\begin{example}
The set of signed arc permutation $\mathcal{A}_3^s$ is equal to
\begin{align*}
\mathcal{A}_3^s = &123 \cup 312 \cup 231\\
			  & \cup (12\shuffle\text{-}3)\cup (31\shuffle\text{-}2)\cup (23\shuffle\text{-}1)\\
			  &\cup(1\shuffle\text{-}3\text{-}2)\cup(2\shuffle\text{-}1\text{-}3)\cup(3\shuffle\text{-}2\text{-}1)\\
			  &\cup\text{-}3\text{-}2\text{-}1\cup\text{-}2\text{-}1\text{-}3\cup\text{-}1\text{-}3\text{-}2
\end{align*}
\end{example}
We have the following corollary.
\begin{corollary}
A signed arc permutation $\pi$ of $\mathcal{A}^s_n$ has a descent in position $i>0$ if and only if:
\begin{itemize}
\item $\pi_{i} > 0$ and $\pi_{i+1} < 0$ or
\item $\pi_{i} = n$ and $\pi_{i+1} = 1$ or
\item $\pi_{i} = \text{-}1$ and $\pi_{i+1} = \text{-}n$. 
\end{itemize}
\end{corollary}  
We use these various properties to split the set of signed arc permutations into six non-overlapping types characterised by their positive and negative subsequences and the sign of their entries with absolute value $1$ and $n$. The six types are defined in the table of Figure~\ref{fig : SAPermutations} along with a graphical description indicating the positive and negative subsequences. 

\begin{figure} [h!]
\begin{center}
\begin{tabular}{| >{\centering\arraybackslash}m{3cm} |>{\centering\arraybackslash}m{9cm} |}
\hline
 \resizebox{!}{3cm}{
\begin{tikzpicture}[line cap=round,line join=round,>=triangle 45,x=1.0cm,y=1.0cm]
\draw [line width=1.pt] (0.,0.) circle (1.cm);

\foreach \x in {0,30,60,90,120,150,180,210,240,270,300,330}
\draw [fill=red] (\x:1cm) circle (1.2pt); 
\draw (180:1.2cm) node {$n$};
\draw (150:1.2cm) node {$1$};
\draw (120:1.2cm) node {$2$};

\draw (30:1.2cm) node {$k$};
\draw[thick] (0:0.8) arc (0:22:0.8);
\draw[thick, <-] (38:0.8) arc (38:360:0.8);
\draw (210:0.6cm) node {$+$};
\draw (0, -1.5) node {Type 1};
\end{tikzpicture}
}
 &  $\bigcup_{2 \leq k \leq n} \left \{\text{\phantom{-}}k\dots n \text{\phantom{-}}1 \dots k-1\right \} ~\bigcup~ 1 \dots n$\\  \hline
\resizebox{!}{3cm}{
\begin{tikzpicture}[line cap=round,line join=round,>=triangle 45,x=1.0cm,y=1.0cm]
\draw [line width=1.pt] (0.,0.) circle (1.cm);

\foreach \x in {0,30,60,90,120,150,180,210,240,270,300,330}
\draw [fill=red] (\x:1cm) circle (1.2pt); 
\draw (180:1.2cm) node {$n$};
\draw (150:1.2cm) node {$1$};
\draw (120:1.2cm) node {$2$};

\draw (30:1.2cm) node {$k$};
\draw[thick, ->] (0:0.8) arc (0:22:0.8);
\draw[thick] (38:0.8) arc (38:360:0.8);
\draw (210:0.6cm) node {$-$};
\draw (0, -1.5) node {Type 2};
\end{tikzpicture}
}
 &  $\bigcup_{2\leq k \leq n}\left \{\text{-}(k-1) \dots \text{-}1 \text{-}n \dots \text{-}k \right \} ~\bigcup~ \text{-}n \dots \text{-}1$  \\ \hline
 \resizebox{!}{3cm}{
\begin{tikzpicture}[line cap=round,line join=round,>=triangle 45,x=1.0cm,y=1.0cm]
\draw [line width=1.pt] (0.,0.) circle (1.cm);

\foreach \x in {0,30,60,90,120,150,180,210,240,270,300,330}
\draw [fill=red] (\x:1cm) circle (1.2pt); 
\draw (180:1.2cm) node {$n$};
\draw (150:1.2cm) node {$1$};
\draw (120:1.2cm) node {$2$};
\draw (30:1.2cm) node {$l$};
\draw (270:1.2cm) node {$k$};
\draw[thick, ->] (270:0.8) arc (270:360:0.8);
\draw[thick, ->] (240:0.8) arc (240:30:0.8);
\draw (300:0.6cm) node {$-$};
\draw (150:0.6cm) node {$+$};
\draw (0, -1.5) node {Type 3};
\end{tikzpicture}
} &  $\bigcup_{1 \leq l < k < n}\left \{ \text{-}k \text{-}(k-1) \dots \text{-}(l+1) ~\shuffle~ k+1 \dots n \text{\phantom{-}}1 \dots l\right \}$ \vspace{-8pt} \\ \hline
 \resizebox{!}{3cm}{
\begin{tikzpicture}[line cap=round,line join=round,>=triangle 45,x=1.0cm,y=1.0cm]
\draw [line width=1.pt] (0.,0.) circle (1.cm);

\foreach \x in {0,30,60,90,120,150,180,210,240,270,300,330}
\draw [fill=red] (\x:1cm) circle (1.2pt); 
\draw (180:1.2cm) node {$n$};
\draw (150:1.2cm) node {$1$};
\draw (120:1.2cm) node {$2$};
\draw (30:1.2cm) node {$k$};
\draw (270:1.2cm) node {$l$};
\draw[thick, <-] (270:0.8) arc (270:360:0.8);
\draw[thick, <-] (240:0.8) arc (240:30:0.8);
\draw (330:0.6cm) node {$+$};
\draw (150:0.6cm) node {$-$};
\draw (0, -1.5) node {Type 4};
\end{tikzpicture}
} & $\bigcup_{1 \leq k < l < n}\left \{ \text{-}k \text{-}(k-1) \dots \text{-}1 \text{-}n \dots \text{-}(l+1) ~\shuffle~ k+1\dots l \right \}$ \vspace{-8pt} \\ \hline
 \resizebox{!}{3cm}{
\begin{tikzpicture}[line cap=round,line join=round,>=triangle 45,x=1.0cm,y=1.0cm]
\draw [line width=1.pt] (0.,0.) circle (1.cm);

\foreach \x in {0,30,60,90,120,150,180,210,240,270,300,330}
\draw [fill=red] (\x:1cm) circle (1.2pt); 
\draw (180:1.2cm) node {$n$};
\draw (150:1.2cm) node {$1$};
\draw (120:1.2cm) node {$2$};
\draw (30:1.2cm) node {$k$};
\draw[thick, <-] (180:0.8) arc (180:360:0.8);
\draw[thick, <-] (150:0.8) arc (150:30:0.8);
\draw (270:0.6cm) node {$+$};
\draw (90:0.6cm) node {$-$};
\draw (0, -1.5) node {Type 5};
\end{tikzpicture}
} &  $\bigcup_{1 \leq k \leq n-1} \left \{ \text{-} k \text{-}(k-1) \dots \text{-}1 ~\shuffle~ k+1 \dots n \right \}$  \\ \hline
 \resizebox{!}{3cm}{
\begin{tikzpicture}[line cap=round,line join=round,>=triangle 45,x=1.0cm,y=1.0cm]
\draw [line width=1.pt] (0.,0.) circle (1.cm);

\foreach \x in {0,30,60,90,120,150,180,210,240,270,300,330}
\draw [fill=red] (\x:1cm) circle (1.2pt); 
\draw (180:1.2cm) node {$n$};
\draw (150:1.2cm) node {$1$};
\draw (120:1.2cm) node {$2$};
\draw (30:1.2cm) node {$k$};
\draw[thick, ->] (180:0.8) arc (180:360:0.8);
\draw[thick, ->] (150:0.8) arc (150:30:0.8);
\draw (270:0.6cm) node {$-$};
\draw (90:0.6cm) node {$+$};
\draw (0, -1.5) node {Type 6};
\end{tikzpicture}
} &  $\bigcup_{1 \leq k \leq n-1} \left \{ \text{-}n \dots \text{-}(k+1) ~\shuffle~ 1 \dots k \right \}$ \\ \hline
\end{tabular}
\end{center}
\caption{
Our six types of signed arc permutations. The first column provides a graphical representation of the positive and negative subsequences. The second column gives a formal description of the permutations of the given type. 
}
\label{fig : SAPermutations}
\end{figure}
\subsection{Explicit bijections for all types}
We build for each type of signed arc permutations an explicit bijection between the permutations it contains and a tractable set of standard domino tableaux. While we end up with six different bijections (one for each type), our constructions share a common set of rules. \\
Given a signed arc permutation $\pi = \pi_1 \pi_2 \dots\pi_n$ of $B_n$, we recursively build a standard domino tableau. At step $1\leq i \leq n$ we add a domino with label $i$ following one of the rules below. The selected rule depends on the type of $\pi$ and the value of $\pi(i)$. 
\begin{itemize}
\item (Rule 1)  Add a horizontal domino at the end of the first row. 
\item (Rule 2) Add either a horizontal domino at the end of the second row or a vertical domino across the first two rows. Exactly one of these two options is possible whenever we use Rule 2.
\item (Rule 3) Add either a horizontal domino at the end of the third row or a vertical domino across the rows number two and three. Exactly one of these two options is possible whenever we use Rule 3. 
\item (Rule 4) Add a vertical domino across the rows number three and four.
\end{itemize}
The next step is to use these four rules to build a specific bijection for each type of signed arc permutations.\\
We start with types 5 and 6 as they provide the best insight about our method.   
\begin{proposition}
\label{prop : 56}
Both type 5 and type 6 permutations are in descent preserving bijection with the set of standard domino tableaux of shapes $(a,2n-a)$ for $a \ge 2n-a \ge 2$. 
\end{proposition}
\begin{proof}
We give the proof for type 5 permutations but the same reasoning applies to type~6. Type 5 contains the permutations 
$$\bigcup_{1 \leq k \leq n-1} \{\text{-} k \text{-}(k-1) \dots \text{-}1 ~\shuffle~ k+1 \dots n\}.$$
\begin{itemize}
\item Firstly,  map $\pi^0 = \text{-}1\,\,2\dots n$ to the standard domino tableau $T^0$ composed of $n$ vertical dominoes. We have $Des(\pi^0) = Des(T^0)=\{0\}$. 
\item Secondly, let $\pi = \pi_1\pi_2\dots\pi_n \neq \pi^0$. We build recursively a two-row standard domino tableau. At step $1\leq i \leq n$ we add a domino with label $i$ according to Rule 1 (resp. Rule 2) if $\pi_i > 0$ (resp. $\pi_i < 0$). 
\end{itemize}
This mapping is clearly bijective and descent preserving. Indeed, given a standard domino tableau $T$ of shape $(a,2n-a)$, the combined number of horizontal dominoes in the second row and vertical dominoes gives the integer $k$ and subsequently the positive and negative subsequences (that only depend on $k$). Then the position of domino $i$ gives the sign of the $i$-th entry of a preimage of $T$ (positive if $i$ is horizontal on the first row and negative if $i$ is horizontal on the second row or vertical). We recover the particular shuffle of the two subsequences and a unique preimage of $T$.\\ Finally, there is a descent in position $i>0$ in the domino tableau $T$ of preimage $\pi$ if and only if $i$ is in the first row and $i+1$ in the second row, i.e. if and only if $\pi_i >0$ and $\pi_{i+1} < 0$ i.e. if and only if $\pi_i > 0 > \pi_{i+1}$. There is a descent in $0$ in $T$ if and only if $\pi_1 < 0$. As a result the bijection is descent preserving.  
\end{proof}
\begin{example}
The type 5 signed arc permutation $\pi = \text{-}5\,\,6\, \text{-}4\,\,7\, \text{-}3\, \text{-}2\, \text{-}1$ is mapped to the following standard domino tableau:
\[
\begin{matrix}
\resizebox{!}{1.24 cm}{%
\begin{tikzpicture}[node distance=0 cm, outer sep = 0pt, line width=1pt]
        \node[bver] (1) at ( 0,   -0.5) {\bf \Large 1};
        \node[bhor] (2) at (1.5, 0) {\bf \Large 2};
        \node[bhor] (3) [below = of 2] {\bf \Large 3};
        \node[bhor] (4) [right = of 2] {\bf \Large 4};
        \node[bhor] (5) [below = of 4] {\bf \Large 5};
        \node[bver] (6) at ( 5,   -0.5) {\bf \Large 6};    
	\node[bver] (7) [right = of 6] {\bf \Large 7};
\end{tikzpicture}  
} \end{matrix}
\]
\end{example}
The following bijections are variations of the one in Proposition \ref{prop : 56}.
\begin{proposition}
\label{prop : 1}
Type 1 signed arc permutations are in descent preserving bijection with the set of standard domino tableaux of shapes $(2n)$ and $(2n-2, 1, 1)$.  
\end{proposition}
\begin{proof}
Type 1 contains the permutations
$$\bigcup_{2 \leq k \leq n} \{k \dots n \phantom{\text{-}}1 \dots k-1\} ~\bigcup~ 1 \dots n.$$
We use the following bijection. First map the permutation $\pi^1 = 1\,2\dots n$ to the unique standard domino tableau $T^1$ of shape $(2n)$. Obviously $Des(T^1) = Des (\pi^1) = \emptyset$. 
Then for $\pi \neq \pi^1$ use the same mapping as for type 5 signed arc permutations except for index $i$ with $\pi_i = 1,~ i>1$. In the latter case apply Rule 3.\\
This construction is clearly bijective. Indeed, type 1 signed arc permutations are determined by the value of the integer $k$. But if $i$ is the index of the vertical domino across the second and the third row of a standard domino tableau of shape $(2n-2, 1, 1)$, then $n-i+1 = k-1$. Finally, the descent set is preserved. Assume $\pi \neq \pi^1$ is a type 1 signed arc permutation mapped to the standard domino tableau $T$.   Let $i>1$ such that $\pi_i = 1$. We have $Des(\pi) = \{i-1\}$. As we apply Rule 3 for index $i$ in our mapping and Rule 1 for all $j\neq i$, only the domino with label $i$ in $T$ is strictly below the others. As a result $Des(T) = \{i-1\} = Des (\pi)$.
\end{proof}
\begin{example}
The signed arc permutation $\pi = 4567123$ is mapped to the following standard domino tableau: 
\[
\begin{matrix}
\resizebox{!}{1.8 cm}{%
\begin{tikzpicture}[node distance=0 cm, outer sep = 0pt, line width=1pt]
        \node[bhor] (1) at (0, 0) {\bf \Large 1};
        \node[bhor] (2) [right = of 1] {\bf \Large 2};
        \node[bhor] (3) [right = of 2] {\bf \Large 3};
        \node[bhor] (4) [right = of 3] {\bf \Large 4};
        \node[bver] (5) at (-0.5, -1.5) {\bf \Large 5};
        \node[bhor] (6) [right = of 4] {\bf \Large 6};    
	\node[bhor] (7) [right = of 6] {\bf \Large 7};
\end{tikzpicture}  
}
\end{matrix} \]
\end{example}
\begin{proposition}
\label{prop : 2}
Type 2 signed arc permutations are in descent preserving bijection with the set of standard domino tableaux of shapes $(2n-1, 1)$ and $(2n-3, 1, 1, 1)$. 
\end{proposition}
\begin{proof}
Type 2 contains the permutations 
$$\bigcup_{2 \leq k \leq n} \{\text{-}(k-1) \dots \text{-}1 \text{-}n \dots \text{-}k\} ~\bigcup~ \{\text{-}n \dots \text{-}1\}.$$
For type 2 we use the following mapping. We map the permutation $\pi^2 = \text{-n} \dots \,\,\text{-}2 \text{-}1$ to the unique standard domino tableau $T^2$ of shape $(2n-1, 1)$. As a result, $Des(\pi^2) = Des(T^2) = \{0\}$. Let $\pi = \pi_1\pi_2\dots\pi_n \ne \pi^2$. We map $\pi$ to the standard domino tableau $T$ built according to the following procedure.
\begin{itemize}
\item The domino labelled $1$ is in vertical position. 
\item Then, apply Rule 1 for all $i$ such that $\pi_i \neq \text{-}n$.
\item Finally, apply Rule 4 for $i$ such that $\pi_i =\text{-}n$. 
\end{itemize}
A type 2 signed arc permutation $\pi \neq \pi^2$ with $\pi_i = \text{-}n$ ($i>1$) has descent set $Des(\pi) =\{0, i-1\}$. As a result, the construction is bijective and descent preserving. Indeed, if $i$ is the label of the vertical domino across the third and fourth row in a standard domino tableau $T$ of shape $(2n-3,1,1,1)$, then set $i = k$ to recover the unique preimage of $T$ by our construction. Furthermore, $Des(T) = \{0, i-1\}$. 
\end{proof}
\begin{example}
The signed arc permutation $\pi = \text{-}3\text{-}2\text{-}1\text{-}7\text{-}6\text{-}5\text{-}4$ is mapped to the following standard domino tableau: 
\[
\begin{matrix}
\resizebox{!}{2.4 cm}{%
\begin{tikzpicture}[node distance=0 cm, outer sep = 0pt, line width=1pt]
        \node[bver] (1) at (-0.5, -0.5) {\bf \Large 1};
        \node[bhor] (2) at (1, 0) {\bf \Large 2};
        \node[bhor] (3) [right = of 2] {\bf \Large 3};
        \node[bver] (4) [below = of 1] {\bf \Large 4};
        \node[bhor] (5) [right = of 3] {\bf \Large 5};
        \node[bhor] (6) [right = of 5] {\bf \Large 6};    
	\node[bhor] (7) [right = of 6] {\bf \Large 7};
\end{tikzpicture}  
}
\end{matrix} \]
\end{example}

\begin{proposition}
\label{prop : 4}
Type 4 signed arc permutations are in descent preserving bijection with standard domino tableaux of shape $(a,2n-a-2,1,1)$ for $a \ge 2n-a-2 \ge 2$. 
\end{proposition}
\begin{proof}
Type 4 contains the permutations 
$$\bigcup_{1 \leq k < l < n} \{\text{-}k \text{-}(k-1) \dots \text{-}1\, \text{-}n \dots \text{-}(l+1) ~\shuffle~ k+1 \dots l\}.$$
We use the same mapping as for type 5 signed arc permutations except for step $i$ such that $\pi_i = \text{-}n,~ i>1$, when we apply Rule 4. As $\text{-}n$ is never the first negative number in $\pi$, the second row of the tableau is always of length greater or equal to $1$ when we apply Rule 4. As a consequence, Rule 4 application is always possible for a valid type 4 permutation.\\
To prove the bijection, let $i$ be the index of the vertical domino across the third and the fourth row in a standard domino tableau $T$ of shape $(a,2n-2-a,1,1)$. Set $k$ to be equal to the number of horizontal dominoes in the second row and vertical dominoes (Rule 2 dominoes) with label strictly less than $i$. Then set $n-l-1$ to be the number of Rule 2 dominoes with index greater than $i$. We recover the positive and negative subsequences of the preimage $\pi$ of $T$. Finally, the position of domino $i$ in $T$ gives the sign of $\pi_i$. Iterating for all $i$ gives the unique shuffle of the two subsequences mapped to $T$ by our construction. Further the descent sets of $\pi$ and $T$ are equal following a similar remark as for Type 5 permutations.  
\end{proof}
\begin{example}
The signed arc permutation $\pi = \text{-}3\,\phantom{\text{-}}4\,\text{-}2\,\text{-}1\,\text{-}7\,\phantom{\text{-}}5\,\text{-}6$ is mapped to the following standard domino tableau: 
\[
\begin{matrix}
\resizebox{!}{2.4 cm}{%
\begin{tikzpicture}[node distance=0 cm, outer sep = 0pt, line width=1pt]
        \node[bver] (1) at (-0.5, -0.5) {\bf \Large 1};
        \node[bhor] (2) at (1, 0) {\bf \Large 2};
        \node[bhor] (3) [below = of 2] {\bf \Large 3};
        \node[bver] (4) at (2.5, -0.5) {\bf \Large 4};
        \node[bver] (5) [below = of 1] {\bf \Large 5};
        \node[bhor] (6) at (4, 0) {\bf \Large 6};    
	\node[bhor] (7) [below = of 6] {\bf \Large 7};
\end{tikzpicture}  
}
\end{matrix} \]
\end{example}

The type 3 case is the most intricate as the descent patterns involved are more complicated. We proceed with the sixth and last bijection.  
\begin{proposition}
\label{prop : 3}
Type 3 signed arc permutations are in descent preserving bijection with standard domino tableaux of shapes $(a,2n-a-2,2)$ for $a \ge 2n-a-2 \ge 2$. 
\end{proposition}
\begin{proof}
Type 3 contains the permutations 
$$\bigcup_{1 \leq l < k < n} \{ \text{-}k \text{-}(k-1) \dots \text{-}(l+1) ~\shuffle~ k+1 \dots n \,\,1 \dots l\}.$$ 
For such permutations we consider the different patterns of existence/absence of negative numbers before $n$, between $n$ and $1$ and after $1$. Note that the absence of negative numbers between $n$ and $1$ implies that $1$ follows $n$ immediately. As type 3 permutations have at least one negative number, there are seven possible patterns. Let $$\pi = \pi_1\dots \pi_b \dots n \dots \pi_{a} \dots \pi_{a_2} \dots \pi_n$$ be a type 3 signed arc permutation where, when applicable, 
\begin{itemize}
\item $b$ is the index of the {\bf last negative integer  before} $n$,
\item $a$ is the index of the  {\bf  first negative integer after} $n$,
\item $a_2$ is the index of the  {\bf second negative integer after} $n$.
\end{itemize}
Additionally, let $i_1 = \pi^{-1}(1)$ be the {\bf index of $\bf 1$} in $\pi$.\\ 
As for the other types, we build recursively a standard domino tableau using Rule 1 (resp. Rule 2) at step $i$ if $\pi_i >0$ (resp. $\pi_i <0$ ) except for some particular values of $i$ that depend on the considered pattern. All these exceptions are listed in Figure~\ref{fig : ExceptionsTable}. See also Figure \ref{fig : SAPType3} for a graphical illustration. 
\begin{figure} [h!]
\begin{center}
\begin{tabular}[t]{|c|c|m{20em}|}
\hline
 Pattern & Description & Exceptions \\
\hline 
 1 & $-n-1-$ &  Apply Rule 2 at step $b+1$ although $\pi_{b+1} >0$. Apply Rule 3 at step $a$. \\   \hline 
 2 & $-n-1+$ & Apply Rule 2 at step $b+1$ and step $i_1$. Apply Rule 3 at step $a$. \\   \hline
 3 & $-n+1-$ & Apply Rule 2 at step $i_1$ and Rule 3 at step $a$. \\   
\hline
 4 & $-n+1+$ & Apply Rule 2 at step $b+1$ and Rule 3 at step $i_1$. \\ \hline
 5 & $+n-1-$ & Apply Rule 3 at step $a$ and step $a_2$. \\ \hline
 6 & $+n-1+$ & If $a_2 < i_1$ apply Rule 3 at step $a$ and step $a_2$. Apply Rule 2 at step $i_1$. Otherwise, apply Rule 3 at step $a$ and step $i_1$. \\ \hline
 7 & $+n+1-$ & Apply Rule 2 at step $i_1$ and Rule 3 at step $a$. \\ 
\hline
\end{tabular}
\end{center}
\caption{
Table of exceptions for the case of type 3. The patterns $\pm n \pm 1 \pm$ indicate the existence or absence of at least one negative integer before $n$, between $n$ and $1$ and after $1$. A '$+$' sign means the absence of such a negative integer while a '$-$' sign indicates its presence. Recall that the absence of negative numbers between $n$ and $1$ implies that $1$ follows $n$ immediately. 
}
\label{fig : ExceptionsTable}
\end{figure}

Let $T$ be a standard domino tableau of shape $(a,2n-a-2,2)$. We show that $T$ has a unique type 3 preimage $\pi$ that we recover by the following procedure. First we determine the pattern that $\pi$ follows. 
\begin{enumerate}
\item The tableau $T$ has either two vertical dominoes across the second and third row or a horizontal domino in the third row. In the former case go to item $2$. Otherwise, go to item $3$.  
\item The signed arc permutation $\pi$ follows pattern 5 or 6. If the entries of the two vertical dominoes across the second and third row are not adjacent integers then $\pi$ follows {\bf pattern 5}. 
Otherwise, we consider the location of the next domino. 
If it lies in the first row, then $\pi$ also follows {\bf pattern 5}. 
Otherwise, $\pi$ follows {\bf pattern 6}. 
\item The signed arc permutation $\pi$ follows pattern 1, 2, 3, 4 or 7 and there is exactly one horizontal domino in the third row. 
Denote by $i$ the label of this domino and consider the subtableau $T'$ of $T$ composed of the dominoes labelled $1, 2 \dots i$. If $T'$ contains only one (horizontal) domino in its second row go to item $4$. Otherwise, there are at least two of them and go to item $5$.
\item The signed arc permutation $\pi$ follows {\bf pattern 7}.
\item Let $j$ be the greatest label amongst Rule 2 dominoes (horizontal in the second row or vertical across the first and the second row) of $T'$. If domino $j-1$ is a Rule 1 one (horizontal in the first row) go to item 6. Otherwise, $j-1$ is a Rule 2 domino and go to item 7.
\item The signed arc permutation $\pi$ follows {\bf pattern 3}.
\item If there is no Rule 2 domino in $T \setminus T'$ go to item 8. Otherwise, go to item 9.
\item The signed arc permutation $\pi$ follows {\bf pattern 4}.
\item If $Des(T \setminus T') = \emptyset$, $\pi$ follows {\bf pattern 2}. Otherwise, it follows {\bf pattern 1}.
\end{enumerate}
\begin{figure}[h]
\begin{center}
$$
\begin{matrix}
\text{Pattern 1}&
\begin{matrix}
\overbrace{* \dots *}^{n_1 \ge 0} & - & \overbrace{+ \dots + n}^{n_2 \ge 1}  & \overbrace{- \dots -}^{n_3 \ge 1} & 1 & \overbrace{+ \dots +}^{n_4 \ge 0} & - & \overbrace{* \dots *}^{n_5 \ge 0} \\
R \dots R & R & R2 \dots R R & R3 \dots R & R & R \dots R & R & R \dots R 
\end{matrix}
\vspace{3pt}\\
\text{Pattern 2} &
\begin{matrix}
\overbrace{* \dots *}^{n_1 \ge 0} & - & \overbrace{+ \dots + n}^{n_2 \ge 1}  & \overbrace{- \dots -}^{n_3 \ge 1} & 1 & \overbrace{+ \dots +}^{n_4 \ge 0}\\
R \dots R & R & R2 \dots R R & R3 \dots R & R2 & R \dots R 
\end{matrix}
\vspace{3pt}\\
\text{Pattern 3} &
\begin{matrix}
\overbrace{* \dots *}^{n_1 \ge 0} & n & 1 & \overbrace{+ \dots +}^{n_2 \ge 0} & - & \overbrace{* \dots *}^{n_3 \ge 0}\\
R \dots R & R & R2 & R \dots R & R3 & R \dots R
\end{matrix}
\vspace{3pt}\\
\text{Pattern 4} &
\begin{matrix}
\overbrace{* \dots *}^{n_1 \ge 0} & - & \overbrace{+ \dots + n}^{n_2 \ge 1} & 1 & \overbrace{+ \dots +}^{n_3 \ge 0}\\
R \dots R & R & R2 \dots R R & R3 & R \dots R 
\end{matrix}
\vspace{3pt}\\
\text{Pattern 5} &
\begin{matrix}
\overbrace{+ \dots +}^{n_1 \ge 0} & n & \overbrace{- - \dots -}^{n_2 \ge 2} & 1 & \overbrace{+ \dots +}^{n_3 \ge 0} & - & \overbrace{* \dots *}^{n_4 \ge 0}\\
R \dots R & R & R3 R3 \dots R & R & R \dots R & R & R \dots R \vspace{3pt}\\
\overbrace{+ \dots +}^{n_1 \ge 0} & n & \overbrace{-}^{n_2 = 1} & 1 & \overbrace{+ \dots +}^{n_3 \ge 0} & - & \overbrace{* \dots *}^{n_4 \ge 0}\\
R \dots R & R & R3 & R & R \dots R & R3 & R \dots R 
\end{matrix}
\vspace{3pt}\\
\text{Pattern 6} &
\begin{matrix}
\overbrace{+ \dots +}^{n_1 \ge 0} & n & \overbrace{- - \dots -}^{n_2 \ge 2} & 1 & \overbrace{+ \dots +}^{n_3 \ge 0}\\
R \dots R & R & R3 R3 \dots R & R2 & R \dots R \vspace{3pt}\\
\overbrace{+ \dots +}^{n_1 \ge 0} & n & \overbrace{-}^{n_2 = 1} & 1 & \overbrace{+ \dots +}^{n_3 \ge 0}\\
R \dots R & R & R3 & R3 & R \dots R  
\end{matrix}
\vspace{3pt}\\
\text{Pattern 7} &
\begin{matrix}
\overbrace{+ \dots + n}^{n_1 \ge 1} & 1 & \overbrace{+ \dots +}^{n_2 \ge 0} & - & \overbrace{* \dots *}^{n_3 \ge 0}\\
R \dots R R & R2 & R \dots R & R3 & R \dots R 
\end{matrix}
\end{matrix}
$$
\end{center}
\caption{
Graphical representation of the tableau construction rules for Type 3 signed arc permutations of each pattern. The first row explicits the template of the permutation of the considered pattern. A '$+$' indicates a positive integer, a '$-$' a negative one, a '$*$' means that the entry is either positive or negative. The second row indicates the applicable rule of construction of the domino tableau. We write $R$ when we use Rule 1 (resp. Rule 2) for positive (resp. negative) entries and $R2, R3$ to indicate the use of Rule 2 and Rule 3 for the exceptions listed in Figure \ref{fig : ExceptionsTable}. Dots in the second row indicate the use of the normal rule $R$. 
}
\label{fig : SAPType3}
\end{figure}
Secondly, as the pattern of $\pi$ is identified, we recover a unique permutation following the templates in Figure \ref{fig : SAPType3}. 
\begin{itemize}
\item For each step $i$ we know the sign of $\pi_i$ (including exceptional steps) and recover the pattern of shuffle of the positive and negative subsequences. 
\item Rule 3 dominoes allow us to compute the position $j$ for which $\pi_j = 1$. The number of positive $\pi_i$, such that $i < j$ determines $n-k$. We recover $l$ by looking at the positive  $\pi_i$, such that $j < i$. The two subsequences are determined. 
\end{itemize}
Finally, as the two subsequences and the shuffle pattern are determined so is $\pi$. 
One can note that the bijection is descent preserving by looking at Figure \ref{fig : SAPType3}. The application of either the common rule $R$ or the exceptions $R2$ and $R3$ guarantees the simultaneous existence or absence of descent in both the permutation and the tableau.
\end{proof}

\begin{example}
The type 3 (pattern 2) signed arc permutation $\pi = \text{-}4\,\text{-}3\,\,5\,\,6\,\,7\,\text{-}2\,\,1$ is mapped to the following standard domino tableau: 
\[
\begin{matrix}
\resizebox{!}{1.8 cm}{%
\begin{tikzpicture}[node distance=0 cm, outer sep = 0pt, line width=1pt]
        \node[bver] (1) at (-0.5, -0.5) {\bf \Large 1};
        \node[bver] (2) [right = of 1] {\bf \Large 2};
        \node[bver] (3) [right = of 2] {\bf \Large 3};
        \node[bhor] (4) at (3, 0) {\bf \Large 4};
        \node[bhor] (5)  [right = of 4] {\bf \Large 5};
        \node[bhor] (6) at (0, -2) {\bf \Large 6};    
	\node[bhor] (7) [below = of 4] {\bf \Large 7};
\end{tikzpicture}  
}
\end{matrix}  \]
\end{example}
We finish the proof of Theorem~\ref{thm:SAP} using the descent preserving bijections above to write

\begin{align*}
\sum_{\pi \in \mathcal{A}^s_{n}} F^B_{Des(\pi)} &= \sum_{T \in SDT(2n)} F^B_{Des(T)}+\sum_{T \in SDT(2n-2,1,1)} F^B_{Des(T)} \\
&+\sum_{T \in SDT(2n-1,1)} F^B_{Des(T)}+\sum_{T \in SDT(2n-3,1,1,1)}\hspace{-5mm} F^B_{Des(T)}\\
&+\sum_{a,T \in SDT(a,2n-a-2,2)} F^B_{Des(T)}+\sum_{a,T \in SDT(a,2n-a-2,1,1)} F^B_{Des(T)} \\
&+2\sum_{a,T \in SDT(a,2n-a)} F^B_{Des(T)},
\end{align*}
which gives Theorem~\ref{thm:SAP} after application of Equation (\ref{eq : GdX}).
\section{Alternative proof of Theorem~\ref{thm:SAP}}
\label{sec : AltProof}
Although in the previous sections our development based on domino tableaux and Chow's quasisymmetric functions fully proves Theorem~\ref{thm:SAP}, the connection between our work and the approach in \cite{AdiAthEliRoi15} remains a natural question. 
As stated in introduction, their approach involves bi-tableaux and Poirier's quasisymmetric functions. In this section we connect our work to these objects and provide an alternative proof of Theorem~\ref{thm:SAP}. While this new proof is somewhat simpler, the bijections involved are less explicit. We briefly introduce the necessary definitions and refer the reader to \cite{AdiAthEliRoi15} for more details.   

\subsection{Signed descent set and Poirier's quasisymmetric functions}

We denote $sDes(\pi)$ the {\bf signed descent set}  of $\pi \in B_n $ defined as $sDes(\pi) = (S,\eps)$ with
\begin{itemize}
\item 
$
S = \{n\}\cup \{ i \mid \begin{cases}  \pi(i) > \pi(i+1), & \mbox{ if } \pi(i) >0 \\   \mbox{either } \pi(i+1) >0 \mbox{ or }  |\pi(i)| > |\pi(i+1)|, & \mbox{ if } \pi(i)<0 \end{cases} \}
$
\item $\eps$ is the mapping from $S$ to $\{-,+\}$ defined as $\eps(s) = +$ if $\pi(s)>0$ and $\eps(s) = -$ otherwise. The sign vector $\eps$ is extended to $[n]$ by setting $\tilde{\eps}(j) = \eps(s_i)$ for $s_{i-1}<j\leq s_i$ (assuming $s_0 = 0$).
\end{itemize}

Further denote 
$$wDes((S, \eps)) = \{s_k \in S ~|~ s_k \ne n,~ \tilde{\eps}(s_i)\tilde{\eps}(s_{i}+1) \in \{++, --, +-\} \}.$$

\begin{defn}[Poirier's fundamental quasisymmetric functions \cite{Poi98}]
Let $X =\{x_1,x_2,\dots\}$ and $Y =\{y_1,y_2,\dots\}$ be two totally ordered alphabets of commutative indeterminates. In \cite{Poi98} Poirier introduces a new family of quasisymmetric functions that generalises the quasisymmetric functions of Gessel to coloured permutations. In the case of 2-colour (signed) permutations, Poirier's fundamental quasisymmetric function $F^P_{(S,\eps)}$ indexed by the signed set $(S,\eps)$ is defined as 
$$
F^P_{(S,\eps)}(X, Y) = \sum_{j \in wDes((S,\eps)) \Rightarrow i_j < i_{j+1}} z_{i_1} \dots z_{i_n},
$$
where $i_1 \le \dots \le i_n$ and $z_i = x_i$ if $\tilde{\eps}(i) = -$ and $z_i = y_i$ otherwise.
\end{defn}
\begin{defn}
For a positive integer $n$ consider the total order $<_r$ defined as $\text{-}1 <_r \text{-}2 <_r \dots \text{-}n <_r 0 <_r  1 <_r 2 <_r \dots <_r n$. Given $\pi \in B_n$ define 
$$
Des_r(\pi) = \{0 \leq i \leq n-1 \mid \pi(i) >_r \pi(i+1)\}
$$ the notion of descent set associated with the order $<_r$.
\end{defn}
\begin{lem}
\label{rem : reordering}
Let $X =\{x_0, x_1,x_2,\dots\}$ and $X^*=X\setminus\{x_0\}$. Chow's quasisymmetric function indexed by $Des_r(\pi)$ for $\pi \in B_n$ may be expressed in terms of Poirier's quasisymmetric functions as follows:
$$
F^B_{Des_r(\pi)}(X) = F^P_{sDes(\pi)}(X^*, X).
$$
\end{lem} 
\begin{proof}
Note that for $\pi \in B_n$,
$
Des_r(\pi) =  \begin{cases}
   wDes(\pi), & \text{ if } \tilde{\eps}(1) = +, \\
   wDes(\pi) \cup \{0\}, & \text{ if } \tilde{\eps}(1) = -. 
\end{cases} 
$
\end{proof}

\begin{lem}
\label{lem : W1}
There is a bijection $\phi_1 : B_n \rightarrow B_n$ satisfying for all $\pi \in B_n$ 
$$ Des(\pi) = Des_r(\phi_1(\pi));~~~ Des_r(\pi) = Des(\phi_1(\pi)).$$
\end{lem}
\begin{proof}
Let $\pi$ be a signed permutation of $B_n$.
For $s \in [n-1]$ consider $\tilde{\eps}(s)$ and $\tilde{\eps}(s+1)$. If at least one of them is positive then $s$ belongs to $Des_r(\pi)$ if and only if $s$ belongs to $Des(\pi).$ If both $\tilde{\eps}(s)$, $\tilde{\eps}(s+1)$ are negative then $s$ lies exactly in one of these sets. Using this property we suggest the following appropriate bijection $\phi_1$. Let $\pi$ be a signed permutation of $B_n$ and denote $Neg(\pi) = \{1\leq i \leq n | \pi(i)<0\}.$ Write the elements of $Neg(\pi) = \{j_1, j_2 \dots j_{|Neg(\pi)|}\}$ in such a way that $\pi(j_1) < \pi(j_2) < \dots < \pi(j_{|Neg(\pi)|})$
and denote $\al$ the involution on $Neg(\pi)$ defined for $1 \leq k \leq |Neg(\pi)|$ as $\al(j_k) = j_{|Neg(\pi)|+1 - k}.$
Then define the signed permutation $\phi_1(\pi)$ by \linebreak
$\phi_1(\pi)(i) = \begin{cases}
   \pi(i), & \text{ if } \pi(i) >0, \\
   \pi(\al(i)), & \text{ if } \pi(i) <0. 
\end{cases}$
\end{proof}
\begin{example} 
For $\pi = \text{-}3\,\phantom{\text{-}} 8\,\phantom{\text{-}} 5\,\text{-} 2\,\phantom{\text{-}} 1\,\text{-} 9\, \text{-} 7\,\phantom{\text{-}} 4\,\phantom{\text{-}} 6$ 
we have $\phi_1(\pi) = \text{-} 7\,\phantom{\text{-}} 8\,\phantom{\text{-}} 5\,\text{-} 9\,\phantom{\text{-}} 1\,\text{-} 2\,\text{-} 3\,\phantom{\text{-}} 4\,\phantom{\text{-}} 6$.
\end{example} 

\subsection{Bi-tableaux}

Now we emphasise the connection between bi-tableaux and both signed permutations and domino tableaux.\\
\begin{defn}
A {\bf standard bi-tableau} $(T_1, T_2)$ is an ordered pair of Young diagrams of bi-shape $(\la_1, \la_2)$ with $|\la_1|+ |\la_2| =n$ and whose boxes are filled with the elements of $[n]$ such that the entries of both $T_1$ and $T_2$ are strictly increasing along the rows and down the columns. We denote $SBT(\la_1,\la_2)$ the set of standard bi-tableaux of bi-shape $(\la_1, \la_2)$.\\
The {\bf signed descent set} $sDes(T_1, T_2)$ of a standard bi-tableau $(T_1, T_2) \in SBT(\la_1, \la_2)$ is the signed set $(S, \eps)$ defined as follows:
\begin{itemize}
\item S contains all $s \in [n-1]$ for which either both $s$ and $s+1$ appear in the same tableau and $s+1$ is in a lower row than $s$, or $s$ and $s+1$ appear in different tableaux.
\item S contains $n$.
\item For every $s \in S$ we denote $\eps(s) = -$ if $s$ appears in $T_1$ and $\eps(s) = +$ otherwise.
\end{itemize} 
\end{defn}
Further we provide the following equivalent definition of~$Des_r$ for bi-tableaux.

\begin{defn}
The descent set $Des_r(T_1, T_2)$  of a standard bi-tableau $(T_1, T_2) \in SBT(\la, \mu)$ is the subset of $\{0\} \bigcup [n-1]$ containing:
\begin{itemize}
\item each $s \in [n-1]$ such that both $s$ and $s+1$ appear in the same tableau and $s+1$ is in lower row than $s$
\item each $s \in [n-1]$ such that $s \in T_2$ and $s+1 \in T_1$
\item 0 if $1 \in T_1$.
\end{itemize}
\end{defn}
\begin{defn}
A {\bf semistandard bi-tableau} $(T_1, T_2)$ is an ordered pair of Young diagrams of bi-shape $(\la_1, \la_2)$ whose boxes are filled with {\em non-negative} integers such that the entries are strictly increasing down the columns and non-decreasing along the rows. As an additional constraint {\bf zeroes may only appear in $T_2$}.
\end{defn}
\begin{example} 
Both the standard bi-tableau $(T_1,T_2)$ and the semistandard bi-tableau $(U_1,U_2)$ on Figure \ref{fig : BiT} have bi-shape $((3), (2,2,2))$.\\ Furthermore, $sDes(T_1,T_2) = (\{2,3,4,6,8,9\},+-+-++)$ and $Des_r(T_1,T_2)=\{2,4,8\}$.
\end{example}
\begin{figure}[h!]
$$
\left(T_1=
\begin{matrix} 
\resizebox{!}{0.5cm}{%
\begin{tikzpicture}[node distance=0 cm,outer sep = 0pt]
	      \node[bsq] (1) at (0,  0) {\bf \Large 3};
	      \node[bsq] (2) [right = of 1] {\bf \Large 5};
	      \node[bsq] (3) [right = of 2] {\bf \Large 6};	
\end{tikzpicture}
}
\end{matrix},~ 
T_2 =
\begin{matrix} 
\resizebox{!}{1.5cm}{%
\begin{tikzpicture}[node distance=0 cm,outer sep = 0pt]
	      \node[bsq] (1) at (0,  0) {\bf \Large 1};
	      \node[bsq] (2) [right = of 1] {\bf \Large 2};
	      \node[bsq] (3) [below = of 1] {\bf \Large 4};
	      \node[bsq] (4) [right = of 3] {\bf \Large 8};
	      \node[bsq] (5) [below = of 3] {\bf \Large 7};
	      \node[bsq] (6) [right = of 5] {\bf \Large 9};  
\end{tikzpicture}
}
\end{matrix}
~\right)
~~~
\left(U_1 =
\begin{matrix} 
\resizebox{!}{0.5cm}{%
\begin{tikzpicture}[node distance=0 cm,outer sep = 0pt]
	      \node[bsq] (1) at (0,  0) {\bf \Large 2};
	      \node[bsq] (2) [right = of 1] {\bf \Large 5};
	      \node[bsq] (3) [right = of 2] {\bf \Large 5};	
\end{tikzpicture}
}
\end{matrix},~ 
U_2 =
\begin{matrix} 
\resizebox{!}{1.5cm}{%
\begin{tikzpicture}[node distance=0 cm,outer sep = 0pt]
	      \node[bsq] (1) at (0,  0) {\bf \Large 0};
	      \node[bsq] (2) [right = of 1] {\bf \Large 1};
	      \node[bsq] (3) [below = of 1] {\bf \Large 4};
	      \node[bsq] (4) [right = of 3] {\bf \Large 6};
	      \node[bsq] (5) [below = of 3] {\bf \Large 5};
	      \node[bsq] (6) [right = of 5] {\bf \Large 7};  
\end{tikzpicture}
}
\end{matrix}
~~\right)
$$
\caption{A standard and a semistandard bi-tableau.}
\label{fig : BiT}
\end{figure}
The generalisation of the RS correspondence that maps signed permutations to an ordered pair of standard bi-tableaux is rather straightforward (see e.g. \cite[Section 6.2]{Sta82}). Given a signed permutation $\pi$, the ordered pair of standard bi-tableaux $(P_1,P_2),(Q_1,Q_2)$ is built by applying the RS correspondence to both the positive and negative subsequence of $\pi$ (more precisely by applying its extension to more general words, the {\it RSK correspondence} \cite{Knu70}). It maps the negative subsequence to $(P_1,Q_1)$ and the positive subsequence to $(P_2,Q_2)$. This mapping is descent preserving in the sense that $sDes(\pi) = sDes(Q_1,Q_2)$ and $sDes(\pi^{\text{-}1}) = sDes(P_1, P_2)$ as well as $Des_r(\pi) = Des_r(Q_1,Q_2)$ and $Des_r(\pi^{\text{-}1}) = Des_r(P_1, P_2)$.

\begin{example}
The signed permutation $\pi = \text{-}3~8~5~\text{-}2~1~\text{-}9~\text{-}7~4~\text{-}6$ is associated to the ordered pair of bi-tableaux $(P_1,P_2),(Q_1,Q_2)$ depicted on Figure \ref{fig : RSB}. One can check that $sDes(\pi) = sDes(Q_1,Q_2) = ([9],-++-+--+-)$ and $sDes(\pi^{\text{-}1}) = sDes(P_1,P_2) = \left ([9],+--++--+-\right)$. Finally, $Des_r(\pi) = Des_r(Q_1,Q_2) = \{0,2,3,5,6,8\}$ and $Des_r(\pi^{\text{-}1}) = Des_r(P_1,P_2) = \{1,2,4,5,6,8\}$.
\end{example}
\begin{figure}[h]
\begin{center}
\begin{align*}
\pi = \text{-}3~8~5~\text{-}2~&1~\text{-}9~\text{-}7~4~\text{-}6
~~\xrightarrow[~~~~~~]{}~~\\
&\left(P_1=
\begin{matrix} 
\resizebox{!}{1.5cm}{%
\begin{tikzpicture}[node distance=0 cm,outer sep = 0pt]
	      \node[bsq] (1) at (0,  0) {\bf \Large 2};
	      \node[bsq] (2) [right = of 1] {\bf \Large 6};
	      \node[bsq] (3) [below = of 1] {\bf \Large 3};
	      \node[bsq] (4) [right = of 3] {\bf \Large 7};
	       \node[bsq] (3) [below = of 3] {\bf \Large 9};	
\end{tikzpicture}
}
\end{matrix},~ 
P_2 =
\begin{matrix} 
\resizebox{!}{1.5cm}{%
\begin{tikzpicture}[node distance=0 cm,outer sep = 0pt]
	      \node[bsq] (1) at (0,  0) {\bf \Large 1};
	      \node[bsq] (2) [right = of 1] {\bf \Large 4};
	      \node[bsq] (3) [below = of 1] {\bf \Large 5};
	      \node[bsq] (3) [below = of 3] {\bf \Large 8};		
\end{tikzpicture}
}
\end{matrix}
~\right)
,~
\left(Q_1 =
\begin{matrix} 
\resizebox{!}{1.5cm}{%
\begin{tikzpicture}[node distance=0 cm,outer sep = 0pt]
	      \node[bsq] (1) at (0,  0) {\bf \Large 1};
	      \node[bsq] (2) [right = of 1] {\bf \Large 6};
	      \node[bsq] (3) [below = of 1] {\bf \Large 4};
	      \node[bsq] (4) [right = of 3] {\bf \Large 7};
	       \node[bsq] (3) [below = of 3] {\bf \Large 9};	
\end{tikzpicture}
}
\end{matrix},~ 
Q_2 =
\begin{matrix} 
\resizebox{!}{1.5cm}{%
\begin{tikzpicture}[node distance=0 cm,outer sep = 0pt]
	      \node[bsq] (1) at (0,  0) {\bf \Large 2};
	      \node[bsq] (2) [right = of 1] {\bf \Large 8};
	      \node[bsq] (3) [below = of 1] {\bf \Large 3};
	      \node[bsq] (3) [below = of 3] {\bf \Large 5};		
\end{tikzpicture}
}
\end{matrix}
~~\right)
\end{align*}
\end{center}
\caption{
A signed permutation and the associated ordered pair of bi-tableaux. 
}
\label{fig : RSB}
\end{figure}
\subsection{Connection between domino and bi-tableaux}
\label{sec : BiT}
As stated in Remark \ref{rem : Bsym}, there is a well known weight preserving bijection between domino tableaux and bi-tableaux. We give some more details in this section. 
Denote $(T^-,T^+)$ the bi-tableau associated to a domino tableau $T$. According to the description of \cite[Algorithm 6.1]{CarLec95}, the Young tableaux $T^-$ and $T^+$ are built by filling each box of $T$ (a domino is composed of two boxes) with a '--' or a '+' sign such that the top leftmost box is filled with '--' and two adjacent boxes have opposite signs. 
$T^-$ (resp. $T^+$) is obtained from the subtableau of $T$ composed of the dominoes with top rightmost box filled with '--' (resp. '+'). 
   
\begin{example} Figure \ref{fig : SSDT} shows a semistandard domino tableau and its 2-quotient. 
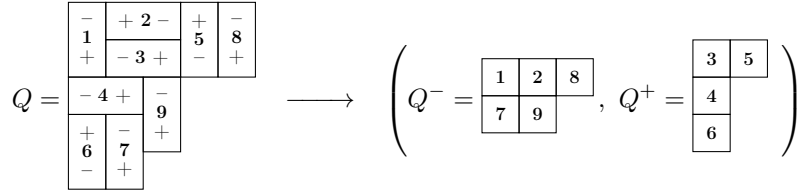
\begin{figure} [h]
\begin{center}
\[
Q =
\begin{matrix}
\resizebox{!}{2.5cm}{%
\begin{tikzpicture}[node distance=0 cm,outer sep = 0pt]
         \node[bver] (1) at ( 0,   0) {$\begin{matrix} 
\text{\raisebox{2pt}{\bf \Large --}} \\
\mathlarger{\mathlarger{\mathlarger{\bm{1}}}} \\ 
\text{\raisebox{-2pt}{\bf \Large +}} \end{matrix}$};
        \node[bhor] (2) at ( 1.5,   0.5) {\bf \Large +~2~--};
        \node[bhor] (3) [below = of 2] {\bf \Large --~3~+};         
        \node[bhor] (4) at ( 0.5,   -1.5) {\bf \Large --~4~+}; 
        \node[bver] (5) at ( 3,   0)  {$\begin{matrix} 
\text{\raisebox{2pt}{\bf \Large +}} \\
\mathlarger{\mathlarger{\mathlarger{\bm{5}}}} \\ 
\text{\raisebox{-2pt}{\bf \Large --}} \end{matrix}$};        
        \node[bver] (6) at ( 0,   -3) {$\begin{matrix} 
\text{\raisebox{2pt}{\bf \Large +}} \\
\mathlarger{\mathlarger{\mathlarger{\bm{6}}}} \\ 
\text{\raisebox{-2pt}{\bf \Large --}} \end{matrix}$};      
        \node[bver] (7)  [right = of 6] {$\begin{matrix} 
\text{\raisebox{2pt}{\bf \Large --}} \\
\mathlarger{\mathlarger{\mathlarger{\bm{7}}}} \\ 
\text{\raisebox{-2pt}{\bf \Large +}} \end{matrix}$};       
        \node[bver] (8) [right = of 5] {$\begin{matrix} 
\text{\raisebox{2pt}{\bf \Large --}} \\
\mathlarger{\mathlarger{\mathlarger{\bm{8}}}} \\ 
\text{\raisebox{-2pt}{\bf \Large +}} \end{matrix}$};
        \node[bver] (9) at ( 2,   -2) {$\begin{matrix} 
\text{\raisebox{2pt}{\bf \Large --}} \\
\mathlarger{\mathlarger{\mathlarger{\bm{9}}}} \\ 
\text{\raisebox{-2pt}{\bf \Large +}} \end{matrix}$};
\end{tikzpicture}  
}
\end{matrix} 
~~\xrightarrow[~~~~~~]{}~~ 
\left(Q^- =
\begin{matrix} 
\resizebox{!}{1cm}{%
\begin{tikzpicture}[node distance=0 cm,outer sep = 0pt]
	      \node[bsq] (1) at (0,  0) {\bf \Large 1};
	      \node[bsq] (2) [right = of 1] {\bf \Large 2};
	      \node[bsq] (3) [right = of 2] {\bf \Large 8};
	      \node[bsq] (4) [below = of 1] {\bf \Large 7};
	      \node[bsq] (5) [below = of 2] {\bf \Large 9};
\end{tikzpicture}
}
\end{matrix},~ 
Q^+ =
\begin{matrix} 
\resizebox{!}{1.5cm}{%
\begin{tikzpicture}[node distance=0 cm,outer sep = 0pt]
	      \node[bsq] (1) at (0,  0) {\bf \Large 3};
	      \node[bsq] (2) [right = of 1] {\bf \Large 5};
	      \node[bsq] (3) [below = of 1] {\bf \Large 4};
	      \node[bsq] (4) [below = of 3] {\bf \Large 6};
\end{tikzpicture}
}
\end{matrix}~
\right)
\]
\end{center}
\caption{
A semistandard domino tableau and its 2-quotient.
}
\label{fig : SSDT}
\end{figure}
\end{example}

We say that this bijection is {\bf shape preserving} as if $T \in SDT(\la)$ then $(T^-,T^+) \in SBT(\la^-,\la^+)$, i.e. the bi-shape of $(T^-,T^+)$ only depends on the shape of $T$. Obviously it is not descent preserving. 
Figure \ref{fig : AllBij} summarises the known bijections between signed permutations, domino tableaux and bi-tableaux.\\ 
\begin{figure}[htbp]
\begin{center}
 \includegraphics[scale=0.45]{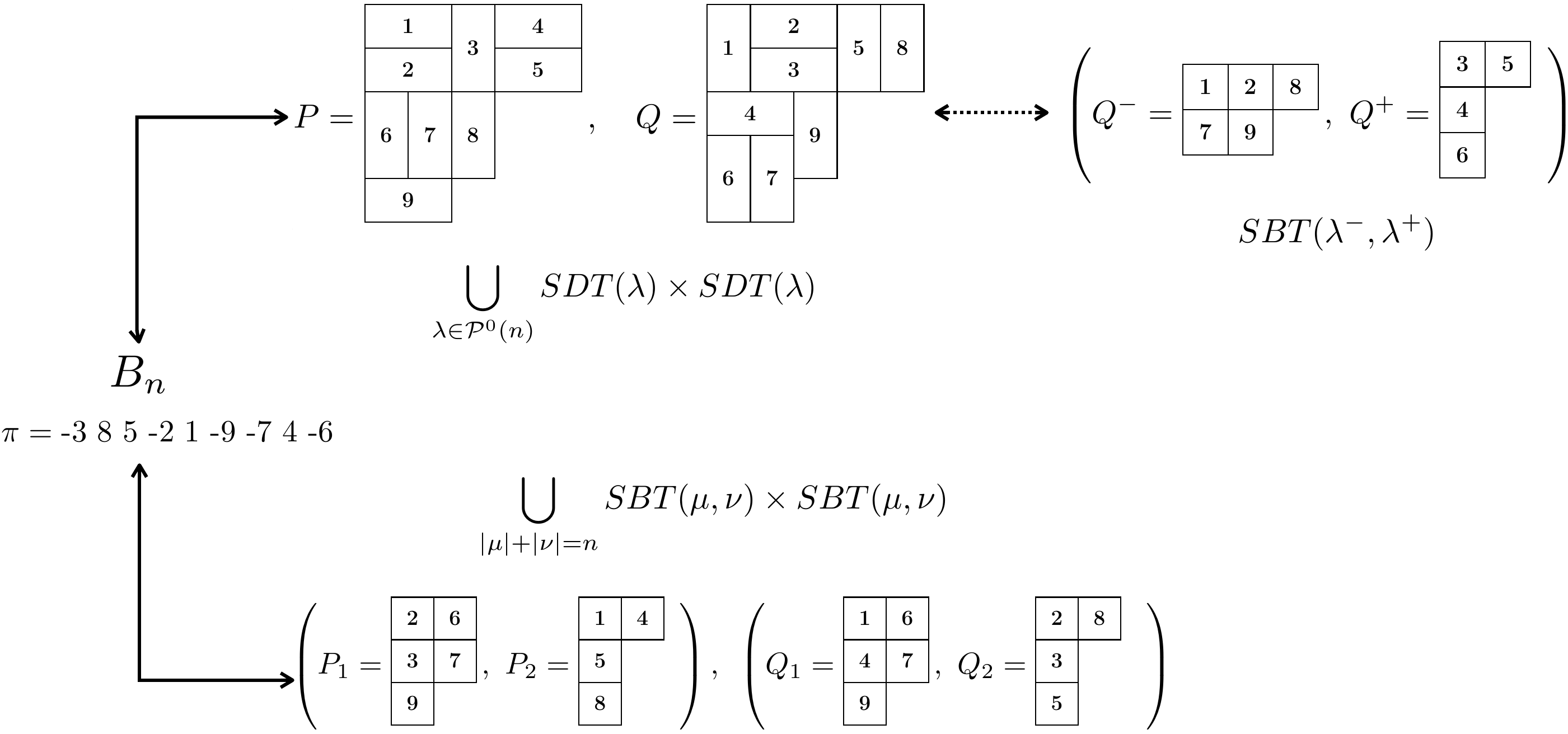}\caption{A summary of the bijections between signed permutations, domino tableaux and bi-tableaux. Plain arrows indicate descent preserving bijections while the dotted one is shape preserving.}
 \label{fig : AllBij}
 \end{center}
 \end{figure}
Using the two type B analogues of the RS correspondence, one may build a descent preserving bijection between domino and bi-tableaux. However, this bijection would not be shape preserving. On the contrary, the Littlewood bijection is shape preserving but not descent preserving. We argue that there is a bijection between standard domino tableaux and standard bi-tableaux that is both descent and shape preserving. This is the object of Lemma \ref{lem : W3}. 
\begin{lem} 
\label{lem : W3} Let $\la \in \mathcal{P}^0(n)$. There is an implicit bijection $\phi_3~:~SBT(\la^-,\la^+) \longrightarrow SDT(\la) $ between standard bi-tableaux of bi-shape $(\la^-,\la^+)$ and standard domino tableaux of shape $\la$, such that for $(T_1,T_2) \in SBT(\la^-,\la^+)$
$$
Des_r((T_1, T_2)) = Des(\phi_3(T_1, T_2)).
$$
\end{lem}
\begin{proof}
Proposition 4.2 in~\cite{AdiAthEliRoi15} states that for all $(\la, \mu)$:
$$
s_\la(X) s_\mu(Y) = \sum_{(T_1, T_2) \in SBT(\la, \mu)} F^P_{sDes((T_1, T_2))}(X, Y). 
$$
Then using Equation \ref{eq : Gss} and Lemma~\ref{rem : reordering} we get 
$$
\mathcal{G}_{\la}(X) = s_{\la^-}(X^*)s_{\la^+}(X)  = \sum_{(T_1, T_2) \in SBT(\la^-, \la^+)} F^B_{Des_r((T_1, T_2))}(X).
$$
On the other hand,
$$
\mathcal{G}_{\la}(X) = \sum_{T \in SDT(\la)} F^B_{Des(T)}(X). 
$$
As the set of type B fundamental quasisymmetric functions is a basis of the ring of type B quasisymmetric functions, the following multisets coincide for any $\la \in \mathcal{P}^0(n)$:
$$
\{Des_r((T_1, T_2))\}_{(T_1, T_2) \in SBT(\la^-, \la^+)} = \{Des(T)\}_{T \in SDT(\la)}.
$$
As a result, there exists a bijection with the required property. 
\end{proof}
\subsection{A second bijection for signed arc permutations}

We use the tools of the previous sections to provide a new proof of Theorem \ref{thm:SAP}. To this end we build a new descent preserving bijection $\phi_2$ between signed arc permutations and bi-tableaux and recover Theorem \ref{thm:SAP} thanks to $\phi_3$. 
More precisely our construction is composed of three mappings. First,  apply bijection $\phi_1$ from Lemma~\ref{lem : W1}. Then we use our new mapping $\phi_2$ that maps the elements of $\phi_1(\mathcal{A}^s_n)$ to standard bi-tableaux. 
The restrictions of $\phi_2$ to each of the six types of signed arc permutations are bijections with standard bi-tableaux whose bi-shapes are the 2-quotients of the shapes in Theorem \ref{thm:SAP}. Finally, we use the bijection $\phi_3$ from Lemma~\ref{lem : W3}.\\
The main ingredient of this section is the following lemma. 
\begin{lem} 
\label{lem : W2} There is a descent preserving mapping $\phi_2$ from the set of signed arc permutations of $\phi_1(\mathcal{A}^s_n)$ to a subset of standard bi-tableaux. The restriction of $\phi_2$ to any of the six types of signed arc permutations is a bijection with standard bi-tableaux whose bi-shapes are the 2-quotients of the shapes in Theorem \ref{thm:SAP}.
\end{lem}
\begin{proof}
The main idea is to build for a signed permutation $\pi$ a bi-tableau $\phi_2(\pi)$ such that the negative elements of $\pi$ correspond to the first Young tableau and the positive elements to the second one. This ensures that signed descents coincide in positions with signs $-+$ and $+-$. Preserving signed descents in positions with signs $--$ ($++$) requires having the descents in the first (second) Young tableau in the same positions as the negative (positive) signed descents in $\pi$. We proceed as follows.\\
Let $\pi = \pi_1\pi_2\dots\pi_n$ be a permutation of $\phi_1(\mathcal{A}^s_{n})$. Begin with the empty bi-tableau $(T_1 = \emptyset, T_2 = \emptyset)$. For $1\leq i \leq n$ build recursively a two-row bi-tableau according to the following procedure. 
\begin{itemize}
\item If $\pi_i>1$ add a square with label $i$ to the first row in $T_2$.
\item If $\pi_i<\text{-}1$ add a square with label $i$ to the first row in $T_1$.
\item If $\pi_i=1$ add a square with label $i$ in $T_2$. If $n \in \pi$ then add it to the second row. Otherwise, add it to the first one. 
\item If $\pi_i=\text{-}1$ add a square with label $i$ in $T_1$. If $\text{-}n \in \pi$ then add it to the second row. Otherwise, add it to the first one.
\end{itemize}

We describe the restriction of $\phi_1(\mathcal{A}^s_{n})$ to each of the six types and the template of the corresponding standard bi-tableaux in Figure \ref{fig : W2} where we denote by $p_{a}$ the position of $\pm a$ in $\phi_1(\pi)$. Clearly, the restriction of $\phi_2$ to each type of signed arc permutations is a descent preserving bijection between the considered set of permutations and the standard bi-tableaux of the corresponding shape. Both the statistics $sDes$ and $Des_r$ are preserved.
\end{proof}
\begin{figure}[h]
\begin{center}
\begin{tabular}[t]{|c|c|}
\hline
type 1 & $\phi_1(\pi) \in \bigcup_{1 < k \leq n} \{k \dots n~1 \dots k-1\} ~\bigcup~ \{1 \dots n\}$ \\[5px]
$\phi_2\phi_1(\pi)$ & $\left (~
\emptyset
~~,~~
\begin{matrix}
\resizebox{!}{1.2cm}{%
\begin{tikzpicture}[node distance=0 cm,outer sep = 0pt]
        \node[bsq] (1) at ( 0,   0) {$p_{k}$};
        \node[bsq] (2) [right = of 1] {$p_{k+1}$};       
		\draw (3,0) node {\bf \LARGE \dots\dots};
        \node[bsq] (6) at ( 5,   0) {$p_{k-1}$};      
        \node[bsq] (7) [below = of 1] {$p_{1}$}; 
\end{tikzpicture}  
}
\end{matrix} ~\right )
~~\text{or}~~
\left (
\emptyset
~~,~~
\begin{matrix}
\resizebox{!}{0.6cm}{%
\begin{tikzpicture}[node distance=0 cm,outer sep = 0pt]
        \node[bsq] (1) at ( 0,   0) {$p_{1}$};
        \node[bsq] (2) [right = of 1] {$p_{2}$};       
		\draw (3,0) node {\bf \LARGE \dots\dots};
        \node[bsq] (6) at ( 5,   0) {$p_{n}$};   
\end{tikzpicture}  
}
\end{matrix}~\right)$ \\
\hline
type 2 & $\phi_1(\pi) \in \bigcup_{1< k \leq n} \{\text{-}k \dots \text{-} n~ \text{-} 1 \dots \text{-}(k-1)\} ~\bigcup~ \{\text{-} 1 \dots \text{-} n\}$ \\[5px]
$\phi_2\phi_1(\pi)$ & $\left (~
\begin{matrix}
\resizebox{!}{1.2cm}{%
\begin{tikzpicture}[node distance=0 cm,outer sep = 0pt]
        \node[bsq] (1) at ( 0,   0) {$p_{k}$};
        \node[bsq] (2) [right = of 1] {$p_{k+1}$};       
		\draw (3,0) node {\bf \LARGE \dots\dots};
        \node[bsq] (6) at ( 5,   0) {$p_{k-1}$};      
        \node[bsq] (7) [below = of 1] {$p_{1}$}; 
\end{tikzpicture}  
}
\end{matrix}
~~,~~ \emptyset
~\right )
~~\text{or}~~
\left (~
\begin{matrix}
\resizebox{!}{0.6cm}{%
\begin{tikzpicture}[node distance=0 cm,outer sep = 0pt]
        \node[bsq] (1) at ( 0,   0) {$p_{1}$};
        \node[bsq] (2) [right = of 1] {$p_{2}$};       
		\draw (3,0) node {\bf \LARGE \dots\dots};
        \node[bsq] (6) at ( 5,   0) {$p_{n}$};   
\end{tikzpicture}  
}
\end{matrix}
~~,~~\emptyset
~\right )$ \\
\hline
type 3 & $\phi_1(\pi) \in \bigcup_{1 \le l < k < n} \{\text{-}(l+1) \dots \text{-} k~ ~\shuffle~~ k+1 \dots n~1 \dots l\}$ \\[5px]
$\phi_2\phi_1(\pi)$ & $\left (~
\begin{matrix}
\resizebox{!}{0.6cm}{%
\begin{tikzpicture}[node distance=0 cm,outer sep = 0pt]
        \node[bsq] (1) at ( 0,   0) {$p_{l+1}$};
        \node[bsq] (2) [right = of 1] {$p_{l+2}$};       
		\draw (3,0) node {\bf \LARGE \dots\dots};
        \node[bsq] (6) at ( 5,   0) {$p_{k}$}; 
\end{tikzpicture}  
}
\end{matrix}
~~,~~ 
\begin{matrix}
\resizebox{!}{1.2cm}{%
\begin{tikzpicture}[node distance=0 cm,outer sep = 0pt]
        \node[bsq] (1) at ( 0,   0) {$p_{k+1}$};
        \node[bsq] (2) [right = of 1] {$p_{k+2}$};       
		\draw (3,0) node {\bf \LARGE \dots\dots};
        \node[bsq] (6) at ( 5,   0) {$p_{l}$}; 
        \node[bsq] (7) [below = of 1] {$p_{1}$};
\end{tikzpicture}  
}
\end{matrix}
~\right )$ \\
\hline
type 4 & 
\footnotesize
$\begin{matrix}
\phi_1(\pi) \in \bigcup_{1 \le k < l < n, k = n-l} \{\text{-}(l+1) \dots \text{-} n ~ \text{-} 1 \dots \text{-} (k) ~~\shuffle~~k+1 \dots l\} \\
\bigcup_{1 \le k < l < n, k > n-l} \{\text{-}(n-l+1) \dots \text{-} k~\text{-} (l+1) \dots \text{-} n ~ \text{-} 1 \dots \text{-} (n-l) ~~\shuffle~~k+1 \dots l\} \\
\bigcup_{1 \le k < l < n, k < n-l} \{\text{-}(n-k+1) \dots \text{-} n~ \text{-} 1 \dots \text{-} k~\text{-} (l+1) \dots \text{-} (n-k) ~~\shuffle~~k+1 \dots l\} \\
\text{Let $\text{-}\al_1, \dots \text{-}\al_{n - l + k}$ be the negative subsequence.}
\end{matrix}
$
\\[5px]
\normalsize
$\phi_2\phi_1(\pi)$ & $\left (~
\begin{matrix}
\resizebox{!}{1.2cm}{%
\begin{tikzpicture}[node distance=0 cm,outer sep = 0pt]
        \node[bsq] (1) at ( 0,   0) {$p_{\al_1}$};
        \node[bsq] (2) [right = of 1] {$p_{\al_2}$};       
		\draw (3,0) node {\bf \LARGE \dots\dots};
        \node[bsq] (6) at ( 5,   0) {$p_{\al_{n-l+k}}$};
        \node[bsq] (7) [below = of 1] {$p_{1}$}; 
\end{tikzpicture}  
}
\end{matrix}
~~,~~ 
\begin{matrix}
\resizebox{!}{0.6cm}{%
\begin{tikzpicture}[node distance=0 cm,outer sep = 0pt]
        \node[bsq] (1) at ( 0,   0) {$p_{k+1}$};
        \node[bsq] (2) [right = of 1] {$p_{k+2}$};       
		\draw (3,0) node {\bf \LARGE \dots\dots};
        \node[bsq] (6) at ( 5,   0) {$p_{l}$};        
\end{tikzpicture}  
}
\end{matrix}
~\right )$ \\
\hline
type 5 & $\phi_1(\pi) \in \bigcup_{1 \leq k < n} \{\text{-} 1\dots \text{-} k~ ~\shuffle~~ k+1\dots n\}$ \\[5px]
$\phi_2\phi_1(\pi)$ & $\left (~
\begin{matrix}
\resizebox{!}{0.6cm}{%
\begin{tikzpicture}[node distance=0 cm,outer sep = 0pt]
        \node[bsq] (1) at ( 0,   0) {$p_{1}$};
        \node[bsq] (2) [right = of 1] {$p_{2}$};       
		\draw (3,0) node {\bf \LARGE \dots\dots};
        \node[bsq] (6) at ( 5,   0) {$p_{k}$}; 
\end{tikzpicture}  
}
\end{matrix}
~~,~~ 
\begin{matrix}
\resizebox{!}{0.6cm}{%
\begin{tikzpicture}[node distance=0 cm,outer sep = 0pt]
        \node[bsq] (1) at ( 0,   0) {$p_{k+1}$};
        \node[bsq] (2) [right = of 1] {$p_{k+2}$};       
		\draw (3,0) node {\bf \LARGE \dots\dots};
        \node[bsq] (6) at ( 5,   0) {$p_{n}$};        
\end{tikzpicture}  
}
\end{matrix}
~\right)$\\
\hline
type 6 & $\phi_1(\pi) \in \bigcup_{1\leq k < n} \{\text{-}(k+1)\dots \text{-} n ~~\shuffle~~ 1\dots k\}$ \\[5px]
$\phi_2\phi_1(\pi)$ & $\left(~
\begin{matrix}
\resizebox{!}{0.6cm}{%
\begin{tikzpicture}[node distance=0 cm,outer sep = 0pt]
        \node[bsq] (1) at ( 0,   0) {$p_{k+1}$};
        \node[bsq] (2) [right = of 1] {$p_{k+2}$};       
		\draw (3,0) node {\bf \LARGE \dots\dots};
        \node[bsq] (6) at ( 5,   0) {$p_{n}$}; 
\end{tikzpicture}  
}
\end{matrix}
~~,~~ 
\begin{matrix}
\resizebox{!}{0.6cm}{%
\begin{tikzpicture}[node distance=0 cm,outer sep = 0pt]
        \node[bsq] (1) at ( 0,   0) {$p_{1}$};
        \node[bsq] (2) [right = of 1] {$p_{2}$};       
		\draw (3,0) node {\bf \LARGE \dots\dots};
        \node[bsq] (6) at ( 5,   0) {$p_{k}$};        
\end{tikzpicture}  
}
\end{matrix}
~\right)$ \\
\hline
\end{tabular}
\end{center}
\caption{Illustration of the mapping $\phi_2$.}
\label{fig : W2}
\end{figure}

Now we are ready to reprove Theorem~\ref{thm:SAP}. The restriction of the mapping $\phi_3 \phi_2 \phi_1$ to each type is a bijection between signed arc permutations of this type and a subset of standard domino tableaux. As proven in Lemmas \ref{lem : W1}, \ref{lem : W3} and \ref{lem : W2} it is descent preserving in the sense that
$$Des(\pi) = Des_r(\phi_1(\pi)) = Des_r(\phi_2\phi_1(\pi)) = Des(\phi_3\phi_2\phi_1(\pi)).$$
Finally, the shapes of the bi-tableaux described in Figure \ref{fig : W2} are exactly the 2-quotients of the shapes from Propositions~\ref{prop : 56}, \ref{prop : 1}, \ref{prop : 2}, \ref{prop : 4} and~\ref{prop : 3}. Since $\phi_3$ is shape preserving, we recover Theorem \ref{thm:SAP}.

\printbibliography
\end{document}